\documentclass[11pt]{amsart}

\usepackage[margin=3cm]{geometry}
\usepackage{amsmath,amsfonts,amssymb,amsthm,graphicx,overpic,hyperref,mathabx}
\usepackage{float,enumitem}
\usepackage[toc,page]{appendix}
\usepackage{hyperref}
\usepackage{graphicx, subcaption}

\newtheorem{thm}{Theorem}[section]
\newtheorem{prp}[thm]{Proposition}

\newtheorem{lem}[thm]{Lemma}

\newtheorem{theostar}{Theorem}

\newtheorem{conj}[theostar]{Conjecture}

\numberwithin{equation}{section}

\newcommand{\nc}{\newcommand}
\nc{\dmo}{\DeclareMathOperator}

\nc{\abs}[1]{\left| #1 \right|}
\nc{\bigO}[1]{\mathcal{O}\left(#1\right)}
\nc{\card}[1]{\left|#1\right|}
\nc{\ceil}[1]{\left\lceil #1 \right\rceil}
\nc{\CC}{\mathbb{C}}
\nc{\floor}[1]{\left\lfloor #1 \right\rfloor}
\nc{\ZZ}{\mathbb{Z}}
\nc{\len}[1]{\left| #1 \right|}
\nc{\littleo}[1]{o\left(#1\right)}
\dmo{\Mat}{Mat}
\nc{\NN}{\mathbb{N}}
\nc{\norm}[1]{\left|\left| #1 \right|\right|}
\nc{\QQ}{\mathbb{Q}}
\nc{\RR}{\mathbb{R}}
\nc{\st}[2]{\left\{ #1 ;\; #2\right\}}
\dmo{\supp}{supp}
\nc{\tr}[1]{\mathrm{tr}\left(#1\right)}

\dmo{\area}{area}
\dmo{\conv}{conv}
\dmo{\diam}{diam}
\dmo{\dist}{\mathrm{d}}
\nc{\HH}{\mathbb{H}}
\dmo{\MCG}{MCG}
\dmo{\MPL}{MPL}
\dmo{\Mod}{\mathcal{M}}
\dmo{\PL}{PL}
\nc{\Sphere}{\mathbb{S}}
\dmo{\sys}{sys}
\dmo{\Teich}{\mathcal{T}}
\nc{\Torus}{\mathbb{T}}
\dmo{\vol}{vol}
\dmo{\WP}{WP}

\dmo{\convTV}{\;\stackrel{\mathrm{TV}}{\longrightarrow}\;}
\nc{\ExV}[2]{\mathbb{E}_{#1}\left[#2\right]}
\dmo{\EE}{\mathbb{E}}
\nc{\Pro}[2]{\mathbb{P}_{#1}\left[#2\right]}
\dmo{\PP}{\mathbb{P}}
\nc{\distTV}[2]{\mathrm{d}_{\rm TV}\left(#1,#2\right)}
\dmo{\UU}{\mathbb{U}}
\nc{\Var}[2]{\mathbb{V}\mathrm{ar}_{#1}\left[#2\right]}

\dmo{\alt}{\mathfrak{A}}
\dmo{\Aut}{Aut}
\dmo{\Fix}{Fix}
\dmo{\Hom}{Hom}
\dmo{\PSL}{PSL}
\dmo{\Rep}{Rep}
\dmo{\sym}{\mathfrak{S}}
\dmo{\inv}{\mathcal{I}}
\dmo{\orb}{\mathcal{O}}
\dmo{\stab}{Stab}
\dmo{\cent}{Z}
\nc{\cox}{\Gamma^\mathrm{Cox}}
\nc{\art}{\Gamma^\mathrm{Art}}
\nc{\eps}{\varepsilon}
\nc{\id}{\mathrm{Id}}
\nc{\bs}{\backslash}
\nc{\isom}{\mathrm{Isom}}

\DeclareFontFamily{U}{wncy}{}
\DeclareFontShape{U}{wncy}{m}{n}{<->wncyr10}{}
\DeclareSymbolFont{mcy}{U}{wncy}{m}{n}
\DeclareMathSymbol{\Loba}{\mathord}{mcy}{"4C} 

\usepackage{ifthen, srcltx}
\usepackage{color}

\newcommand{\showcomments}{yes}

\newsavebox{\commentbox}
\newenvironment{comment}%
{\ifthenelse{\equal{\showcomments}{yes}}%
{\footnotemark
        \begin{lrbox}{\commentbox}
        \begin{minipage}[t]{1.25in}\raggedright\sffamily\tiny
        \footnotemark[\arabic{footnote}]}
{\begin{lrbox}{\commentbox}}}
{\ifthenelse{\equal{\showcomments}{yes}}
{\end{minipage}\end{lrbox}\marginpar{\usebox{\commentbox}}}
{\end{lrbox}}}

\title{Subgroup growth of right-angled Artin and Coxeter groups}

\author{Hyungryul Baik, Bram Petri and Jean Raimbault}
  \address{Department of Mathematical Sciences, KAIST, 291 Daehak-ro Yuseong-gu, Daejeon, 34141, South Korea }
  \email{hrbaik@kaist.ac.kr}
  \address{Mathematical Institute, University of Bonn, Endenicher Allee 60, Bonn, Germany}
  \email{bpetri@math.uni-bonn.de}
  \address{Institut de Math\'ematiques de Toulouse ; UMR5219 \\ Universit\'e de Toulouse ; CNRS \\ UPS IMT, F-31062 Toulouse Cedex 9, France}
  \email{Jean.Raimbault@math.univ-toulouse.fr}

\thanks{
H.~B. was partially supported by Samsung Science \& Technology Foundation grant No. SSTF-BA1702-01.
B.~P. gratefully acknowledges support from the ERC Advanced Grant ``Moduli''. 
J.~R. was supported by the grant ANR-16-CE40-0022-01 - AGIRA. }

\begin{document}

\begin{abstract}
  We determine the factorial growth rate of the number of finite index subgroups of right-angled Artin groups as a function of the index. This turns out to depend solely on the independence number of the defining graph. We also make a conjecture for right-angled Coxeter groups and prove that it holds in a limited setting.
\end{abstract}

\maketitle

\section{Introduction}

\subsection{Subgroup growth}

This paper is a contribution to the topic of {\em subgroup growth}. This is the study of the functions $n \mapsto s_n(\Gamma)$ as $n \to +\infty$, where $\Gamma$ is a finitely generated group and
\[
s_n(\Gamma) = \card{\{ \Gamma' \le \Gamma : |\Gamma/\Gamma'| = n\}}
\]
is the number of subgroups of index exactly $n$ in $\Gamma$. An introduction to the topic, together with a survey of the state of the art at the beginning of the millenium, are given in the book \cite{Lubotzky_Segal} by A.~Lubotzky and D.~Segal. In what follows, we will recall only the parts of the theory which are directly relevant to what we want to do.

The starting point for most results in this area is the immediate relation between subgroups and {\em transitive permutation representations}: if $\Gamma$ contains a subgroup $\Gamma'$ with index $n$ then the action on the left-cosets $\Gamma/\Gamma'$ gives, once they are labeled with $1, \ldots, n$, (with the coset identity labeled 1, but otherwise arbitrarily) a homomorphism from $\Gamma$ to the symmetric groups $\sym_n$ (the group of bijections of the set $\{1, \ldots, n\}$), the image of which acts transitively on $\{1,\ldots,n\}$. In the other direction, if $\Gamma$ acts transitively on $\{1, \ldots, n\}$ then the stabiliser of $1$ is a subgroup of index $n$ in $\Gamma$. The second map is a section of the first, and it is easily seen that the preimage of a subgroup corresponds to the relabelings of $2, \ldots, n$, so there are $(n-1)!$ of them. Thus if we define
\[
t_n(\Gamma) = \card{\{\rho \in \Hom(\Gamma, \sym_n) : \mathrm{Im}(\rho) \text{ is transitive}\}}
\]
we have
\[
s_n(\Gamma) = \frac{t_n(\Gamma)}{(n-1)!}.
\]
Studying $t_n(\Gamma)$ directly is rather hard and usually one  instead considers the total number of permutation representations.
%looks rather at the total number of permutations rather then just the transitive ones. 
Thus let
\[
h_n(\Gamma) = \card{\Hom(\Gamma, \sym_n)}.
\]
In many cases the asymptotic behaviours of $h_n$ and $t_n$ are similar. Let us briefly consider the case of the free groups (which will be useful to us later). Let $\Gamma$ be freely generated by $a_1, \ldots, a_r$. An element in $\hom(\Gamma, \sym_n)$ is just a $r$-tuple of permutations, corresponding to the images of the generators. Thus $h_n(\Gamma) = (n!)^r$, and while it seems hard to directly count the transitive representations, the very fast growth of $h_n(\Gamma)$ together with the fact that an arbitrary representation decomposes into smaller transitive ones allows to prove that the proportion of non-transitive representations is $O(1/n)$ (see \cite[p. 40]{Lubotzky_Segal} or \cite{Dixon_generating} for a more precise result). In particular,  as $n\to\infty$:
\[
s_n(\Gamma) = t_n(\Gamma)/(n-1)! \sim h_n(\Gamma)/(n-1)! = n \cdot (n!)^{r-1}
\]
for the free group\footnote{Recall that we write $f(n)\sim g(n)$ as $n\to\infty$ for two functions $f,g:\NN\to\RR$ to indicate that $f(n)/g(n)\to 1$ as $n\to\infty$.}. When $\Gamma$ is not free one must count tuples of permutations with added constraints and this is much harder. A character-theoretical approach succeeded in getting an asymptotic equivalent for cofinite Fuchsian groups (groups acting properly discontinuously on the hyperbolic plane $\HH^2$ with a fundamental domain of finite volume). In \cite{MP1} T.~M\"uller and J.-C.~Schlage-Puchta deal with surface groups and in \cite{Liebeck_Shalev} M.~Liebeck and A.~Shalev deal with more general Fuchsian groups (free products of finite groups were studied earlier by M\"uller \cite{Mul}). In all these cases the final result takes the form
\[
s_n(\Gamma) \sim F_\Gamma(n) \cdot (n!)^{-\chi(\Gamma)}
\]
where $\chi(\Gamma) < 0$ is the Euler characteristic of the orbifold associated to the Fuchsian group and $1 \le F_\Gamma(n) \le O(C^{\sqrt n})$ for some $C > 1$, which together with Stirling's approximation implies in particular that:
\begin{equation} \label{fuchsian}
  \lim_{n \to +\infty} \frac{\log s_n(\Gamma)}{n\log(n)} = -\chi(\Gamma) = \frac{\vol(\Gamma \bs \HH^2)}{2\pi}. 
\end{equation}
A beautiful application of a slightly different version of this result is the counting of arithmetic surfaces by M.~Belolipetsky--T.~Gelander--A.~Lubotzky--A.~Shalev \cite{BGLS}.

The first step in establishing the asymptotic behaviour of $s_n$ for a given group is thus to study the growth of $s_n(\Gamma)$ at the scale $n^n$. In general we can ask whether the limit 
\begin{equation} \label{growth_rate}
  \lim_{n \to +\infty} \frac{\log s_n(\Gamma)}{n\log(n)} 
\end{equation}
exists, and try to compute it in terms of known quantities associated to $\Gamma$. An example-driven approach to this is given in \cite{Mueller_Puchta_examples}.

The more specific motivation for this paper was to study the limit \eqref{growth_rate} for higher-dimensional hyperbolic lattices, in particular in dimension 3. We will review in some detail what is known in general (following the work of I.~Agol \cite{Agol} and D.~Wise) in \ref{sec_hyp3folds} below, but let us say now that the picture is much wilder and there is no hope of a result as clean as this. In particular there can be no linear relation between covolume and subgroup growth as in \eqref{fuchsian} (see Section \ref{sec_hyp3folds}). We therefore limit our study here to the simplest (in the sense of group presentations) family of known hyperbolic lattices, that of Coxeter groups and more particularly right-angled ones (for which all relations are commutators). Even in this setting we do not reach a complete answer, but we do make a conjecture providing an explicit limit in terms of combinatorial invariants of the Coxeter graph. 

We also study right-angled Artin groups for which the problem is somewhat simpler and for which we can compute the limit in \eqref{growth_rate} in all cases. To finish this overview we mention \cite{ML} in which the subgroup growth of the fundamental groups of certain (non-hyperbolic) 3--manifolds is computed.

%%%%%%%%%%%%%%%%%%%%%%%%%%%%%%%%%%%%%%%%%%%%%%%%%%%%%%%%%%%%

\subsection{Artin groups}

Let $\mathcal G$ be a finite graph: throughout this paper we will represent it by a set of vertices (which we will also denote by $\mathcal G$) together with a symmetric relation $\sim$ which signifies adjacency; to simplify matters we do not allow $v \sim v$ (thus our graphs are combinatorial graphs: they have no loops or mutliple edges). The {\em right-angled Artin group} associated to $\mathcal G$ is the group $\art(\mathcal G)$ given by the presentation
\[
\art(\mathcal G) = \langle \sigma_v : v \in \mathcal G, \sigma_v\sigma_w = \sigma_w\sigma_v \text{ if } v \sim w \rangle
\]
In particular, the complete graph on $r$ vertices corresponds to the free abelian group $\ZZ^r$, and a graph with no edges yields a free group.

To state our result we need to recall the definition of a well-known invariant of finite graphs. An {\em independent set of vertices} in $\mathcal G$ is a subset $S \subset \mathcal G$ such that no two distinct elements of $S$ are adjacent in $\mathcal G$. The {\em independence number} of the graph $\mathcal G$ is the maximal cardinality of such a subset. We will use the notation $\alpha(\mathcal G)$ for it. Our theorem is then stated as follows. 

\begin{theostar} \label{Main_artin}
  Let $\mathcal G$ be a finite graph and $\Gamma = \art(\mathcal G)$. Then:
  \begin{equation} \label{artin_subgroup}
    \lim_{n \to +\infty} \left(\frac{\log s_n(\Gamma)}{n\log n} \right) = \alpha(\mathcal G) - 1.
  \end{equation}
\end{theostar}

%%%%%%%%%%%%%%%%%%%%%%%%%%%%%%

\subsubsection{The lower bound}

There are two steps, of inequal difficulty, in the proof of Theorem \ref{Main_artin}: a lower bound and an upper bound on $s_n(\Gamma)$. The lower bound is
\begin{equation} \label{lb_artin}
  s_n(\Gamma) \ge (n + O(1)) \cdot (n!)^{\alpha(\mathcal G)-1} 
\end{equation}
and it is almost immediately proven from the case of the free group: let $S = \{v_1, \ldots, v_\alpha\}$ be a maximal independent set (that is $\card{S} = \alpha(\mathcal G)$), then the map on generators
\[
\sigma_{v_i} \mapsto a_i, \sigma_v,\, v \not\in S \mapsto \id
\]
defines a surjective morphism from $\Gamma$ to the free group $\Phi_\alpha$ on $a_1, \ldots, a_\alpha$. It follows that $t_n(\Gamma) \ge t_n(\Phi_\alpha)$ and we have seen that $t_n(\Phi_\alpha) \ge (n + O(1)) \cdot (n!)^\alpha$, so we get \eqref{lb_artin}. 

%%%%%%%%%%%%%%%%%%%%%%%%%%%%%%

\subsubsection{The upper bound}

The upper bound is given in Proposition \ref{upper_bd_RAAG}. It is of the form
\begin{equation} \label{upper_bd_artin_intro}
s_n(\Gamma) \le C^{n\log\log n} \cdot (n!)^{\alpha(\mathcal G)-1}
\end{equation}
where $C$ depends (rather transparently, but we will not give a precise statement here) on $\mathcal G$, and the proof is much more involved. We use induction on the number of vertices to prove it. To carry out the induction step one fixes a vertex $v \in \mathcal G$.  Ideally, for $\rho \in \hom(\Gamma, \sym_n)$, whenever $\rho(\sigma_v)$ has few fixed points every $\rho(\sigma_w)$ for $w \sim v$ must have many small orbits and this drastically limits their number; while when $\rho(\sigma_v)$ has many fixed points we can forget it and use the induction hypothesis on the remaining vertices. In practice this vague idea is in default, and we apply a scheme reminescent of it by separating the count according to the size of the orbits of the group $\langle \rho(\sigma_w): w \sim v \rangle$: the three steps (separating, counting with large orbits and counting with small orbits) are given in Lemmas \ref{dec_small_large_lem}, \ref{large_lem} and \ref{small_orbits}.

%%%%%%%%%%%%%%%%%%%%%%%%%%%%%%

\subsubsection{Sharper upper bounds}

There are cases where we get both a much sharper upper bound and a much simpler proof. For example, if $\mathcal G$ is a $2r$-gon then removing one edge every two we get a surjection 
\[
(\ZZ^2)^r \to \art(\mathcal G) 
\]
and since the number of commuting pairs of permutations equals exactly $\Pi(n) \cdot n!$ (where $\Pi(n)$ counts the number of partitions of the integer $n$) we get that $s_n(\art(\mathcal G)) \le n\Pi(n)^r\cdot (n!)^{r-1}$, which is much smaller than the bound given by \eqref{upper_bd_artin_intro} since $\Pi(n) = O(C^{\sqrt n})$. This trick can in fact be applied to all bipartite graphs to get an upper bound of the same order, which we do in Proposition \ref{artin_bipartite}. It does not work in general, for example for a $(2r+1)$-gon it cannot give better than an upper bound $n\Pi(n)^r\cdot (n!)^r$ which is off by a factor of $n!$.

%%%%%%%%%%%%%%%%%%%%%%%%%%%%%%

\subsubsection{A short survey on independence numbers of graphs}

First we note that the independence number has some geometric significance with respect to right-angled Artin groups: the group $\art(\mathcal G)$ is naturally the fundamental group of a cube complex,  called the Salvetti complex, and $\alpha$ counts the maximal number of disjoint hyperplanes in this complex (see for instance \cite[Example 5.2]{Vogtmann}).

The independence number is rather hard to compute in general: the fastest known algorithms have exponential complexity (in the number of vertices). The problem of finding a maximum independent set is equivalent to finding a maximum clique in the complement graph, and a survey of the algorithmic aspects of the latter is given in \cite{survey_maxclique}. For certain classes of sparse graphs such as graphs of bounded degree and planar graphs na\"ive algorithms for the clique problem work in linear time; thus, the computation of the independence number is fast for very dense graphs. 

For various models of random graphs the behaviour of the independence number is rather well-understood. In the Erd\"os--R\'enyi model $\mathcal G_{r, p}$ there is the following result of A. Frieze \cite{Frieze}. Let $\eps > 0$ be fixed, then 
\[
\left| \alpha(\mathcal{G}_{r,p}) - \frac 2 p (\log(rp) - \log\log(rp) - \log 2 + 1) \right| \leq \dfrac{\eps}{p}
\]
with probability going to 1 as $N \to \infty$, if $d_\epsilon/N \leq p = p(r) = o(1)$ for some fixed sufficiently large constant $d_\eps$. For regular random graphs there are upper and lower bounds which are linear in the number of vertices. Tables for the best ones to date are given in \cite{Duckworth_Zito}. In particular for a random trivalent graph $\mathcal G_{r, 3}$ on $r$ vertices we have, almost surely as $r\to +\infty$: 
\[
0.43 \le \frac{\alpha(\mathcal G_{r, 3})} r \le 0.46
\]
and (older) experimental data suggests there might exist $\alpha \cong 0.439$ such that  $\alpha(\mathcal G_{r, 3}) / r  = \alpha + o(1)$ almost surely as $r\to\infty$ \cite{McKay}.

%%%%%%%%%%%%%%%%%%%%%%%%%%%%%%%%%%%%%%%%%%%%%%%%%%%%%%%%%%%%

\subsection{Coxeter groups}

Let $\mathcal G$ be a graph. The {\em right-angled Coxeter group} $\cox(\mathcal G)$ associated to $\mathcal G$ is obtained from the Artin group by adding the condition that every generator be an involution. Namely : 
\[
\cox(\mathcal G) = \langle \sigma_v : v \in \mathcal G, \sigma_v^2 = \id, \sigma_v\sigma_w = \sigma_w\sigma_v \text{ if } v \sim w \rangle. 
\]
We will comment on the geometric significance of this presentation below, for now let us state our conjecture on subgroup growth. For this we need a new invariant $\gamma$ of a graph. Recall that a {\em clique} in $\mathcal G$ is a subset of vertices, any two of which are adjacent to each other (in other words, a complete induced subgraph). We will call an induced subgraph $\mathcal C$ of $\mathcal G$ a {\em clique collection} if each connected component of $\mathcal C$ is a complete graph.

If $\mathcal C_1, \ldots, \mathcal C_q$ are the connected components of a clique collection $\mathcal C$ then we put
\[
w(\mathcal C) = \sum_{i=1}^q \left( 1 - \frac 1 {2^{|\mathcal C_i|}} \right)
\]
and we define $\gamma(\mathcal G)$ as the maximum for $w$:
\begin{equation} \label{def_gamma}
  \gamma(\mathcal G) = \max_{\mathcal C} w(\mathcal C)
\end{equation}
where the maximum is taken over all clique collections in $\mathcal G$. Figure \ref{icosahedron} shows two different collections realising $\gamma(\mathcal G) = 7/4$ for the icosahedral graph. 

\begin{figure}[h] 
  \caption{Two clique collections realising $\gamma$}
  \centering
  \begin{subfigure}{.4\textwidth}
    \includegraphics[width=\textwidth]{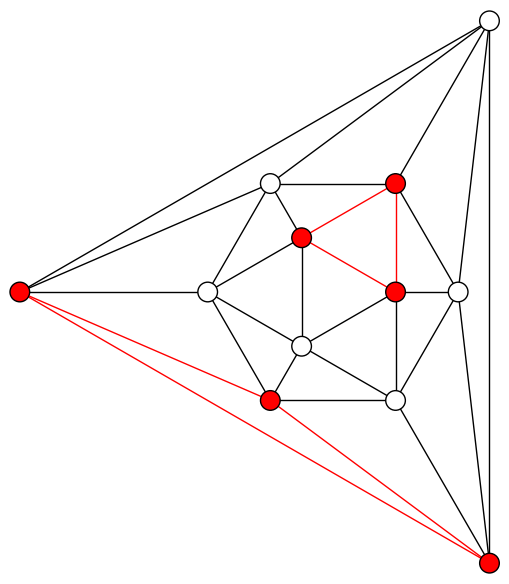}
  \end{subfigure}
  \begin{subfigure}{.4\textwidth}
    \includegraphics[width=\textwidth]{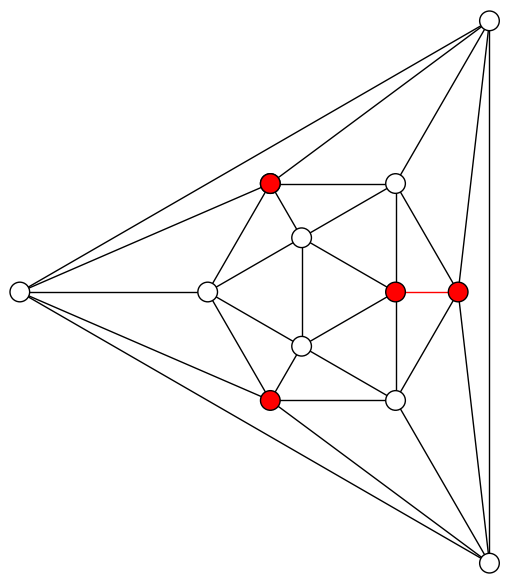}
  \end{subfigure}
  \label{icosahedron}
\end{figure}

Then we conjecture the following. 

\begin{conj} \label{conjecture_cox}
  Let $\mathcal G$ be a finite graph and $\Gamma = \cox(\mathcal G)$. Assume that $\mathcal G$ is not a complete graph (so $\Gamma$ is infinite), then we have : 
  \begin{equation} \label{conjecture_cox_subgroup}
    \lim_{n \to +\infty} \left(\frac{\log s_n(\Gamma)}{n\log n} \right) = \gamma(\mathcal G) - 1. 
  \end{equation}
\end{conj}

We note that the limit inferior in this conjecture holds for all finite graphs, and is proven similarly to the case of Artin group explained above (we will explain this below). Proving that the upper bound holds (if it does) in general seems much harder than in the case of Artin groups; on the other hand it is not hard to construct examples where simple tricks can be used to demonstrate (see \ref{simple_general_const} for a large family of examples, and \ref{sec_octahedron} for a geometric example). We also give a result, whose proof is much harder than the constructions above, establishing the upper bound for $s_n(\cox(\mathcal G))$ under strong structural constraints on the graph $\mathcal G$. It includes the case of trees, for which we also give a simpler proof (see Proposition \ref{trees_directly}). In the theorem below (and in the rest of the paper), $N_1(v)$ denotes the set of neighbours of a vertex $v\in\mathcal G$. The results is as follows. 

\begin{theostar} \label{Main_coxeter}
  Let $\mathcal G$ be a finite graph such that there exist vertices $v_0, \ldots v_l$ and $w_1, \ldots, w_s$ in $\mathcal G$ such that the 1-neighbourhood of each $v_i$ is a tree, as is the 2-neighbourhood of each $w_j$, these neighbourhoods are pairwise disjoint, and the graph
  \[
  \mathcal G \setminus \left( \{v_1, \ldots, v_l \} \cup \bigcup_{j=1}^s N_1(w_j) \right)
  \]
  is a union of trees. Then Conjecture \ref{conjecture_cox} holds for $\mathcal G$. 
\end{theostar}

In practical terms the graphs to which this theorem applies can be constructed as follows: take a tree $\mathcal G_0$ and choose disjoint sets $S_1, \ldots, S_l$ and $T_1, \ldots, T_s$ of leaves of $\mathcal G_0$. Add vertices $v_1, \ldots, v_l$ such that $v_i$ is adjacent to every vertex in $S_i$, and star graphs $\mathcal S_1,  \ldots, \mathcal S_s$ such that every vertex in $T_j$ is adjacent to exactly one leaf of $\mathcal S_j$. Figure \ref{example_thm} pictures two examples of such graphs (with the $v_i$s marked red and the $w_j$s and neighbours green). 

\begin{figure}[h] 
  \caption{Two graphs to which Theorem \ref{Main_coxeter} applies}
  \label{example_thm}
  \centering
  \begin{subfigure}{.45\textwidth}
    \includegraphics[width=\textwidth]{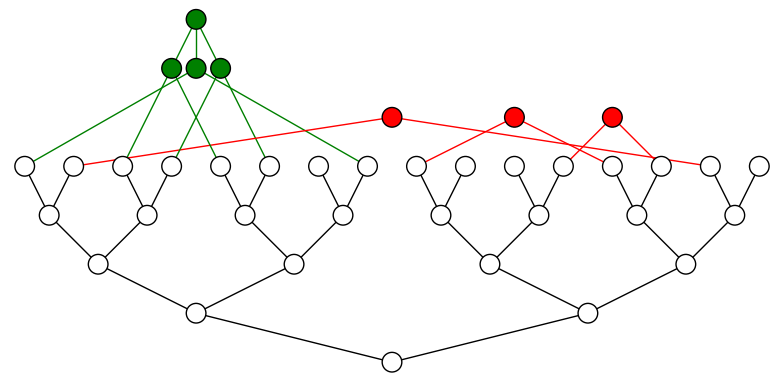}
  \end{subfigure}
  \begin{subfigure}{.45\textwidth}
    \includegraphics[width=\textwidth]{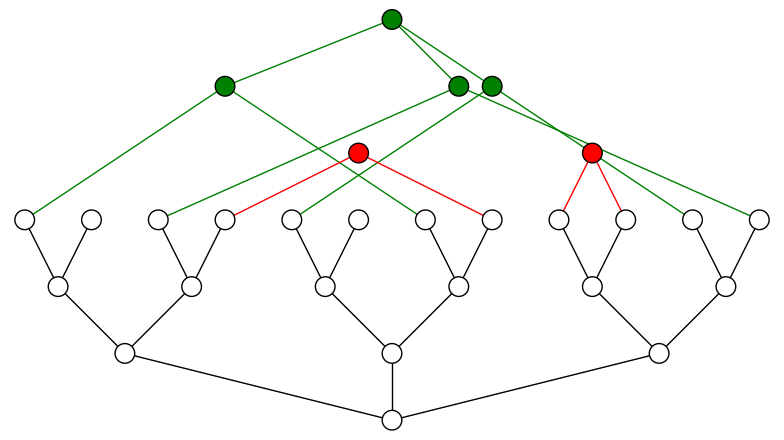}
  \end{subfigure}
\end{figure}

Many graphs do not satisfy the hypotheses of this theorem: for example, a simple obstruction these hypotheses is to have a triangle. There are no 3--dimensional hyperbolic right-angled polyhedra of finite volume whose dual graph does; however we give an elementary proof of the conjecture for a simple example (the ideal octahedron) in \ref{sec_octahedron} below. We note that our result recovers a crude form of Liebeck and Shalev's asymptotic \eqref{fuchsian} for reflection groups associated to right-angled hyperbolic polytopes (see section \ref{Fuchsian_RACG} below); however we actually make use of a particular case of it in our proof. 

%%%%%%%%%%%%%%%%%%%%%%%%%%%%%%

\subsubsection{The lower bound}

Let $\mathcal C = \bigcup_{i=1}^q$ be a clique collection realising $\gamma(\mathcal G)$. Since $\mathcal C$ is an induced subgraph in $\mathcal G$ we get a surjection 
\[
\cox(\mathcal G) \to \cox(\mathcal C) \cong (\ZZ/2\ZZ)^{|\mathcal C_1|} * \ldots * (\ZZ/2\ZZ)^{|\mathcal C_q|},   
\]
where for groups $G,H$ the group $G*H$ denotes their free product. Now as a particular case of M\"uller's result on free products \cite{Mul} (see also \cite[Theorem 14.2.3]{Lubotzky_Segal}, and we give an elementary proof in \ref{prp_cpltegraphsmeth3} below) we get that
\begin{equation} \label{2-gps}
s_n(\cox(\mathcal C)) \sim F(n)\cdot (n!)^{\sum_{i=1}^q \left(1 - \frac 1{2^{|\mathcal C_i|}} \right) - 1}
\end{equation}
where $1 \le F(n) \le C^{\sqrt n}$. 

%%%%%%%%%%%%%%%%%%%%%%%%%%%%%%

\subsubsection{The upper bound}

For the upper bound we would like to proceed in a similar manner as we did for Artin groups. However a major complication arises, which we were not able to overcome without the additional hypotheses on $\mathcal G$ in Theorem \ref{Main_coxeter}. The problem is that when removing a neighbourhood of a vertex $\gamma$ can decrease by 1/2 instead of 1 and this forbids a naive application of the inductive hypothesis. Instead we have to study the representations of the group on the smaller graph with constraints coming from the removed vertices, which we were only able to carry to term for the graphs satisfying the hypotheses in the theorem (see Proposition \ref{coxeter_constraints}). It would be easy to somewhat artificially include more graphs in our statement without complicating much the arguments in the proof (for example we might add new vertices adjacent to the free leaves of the glued star graphs) but we will not do so here as it represents little progress towards Conjecture \ref{conjecture_cox} and we are not aware of further applications. 

We also note that the argument we used for bipartite graphs in the Artin case completely breaks down for Coxeter groups. It gives bounds which are better than the trivial bound $h_n(\cox(\mathcal G)) \le n^{|\mathcal G| n/2}$ but in the general case these upper bounds are still far away from the lower bound.

%%%%%%%%%%%%%%%%%%%%%%%%%%%%%%%%%%%%%%%%%%%%%%%%%%%%%%%%%%%%

\subsection{Layout of the paper}

The further sections in this paper are logically independent from each other. We begin by discussing various examples of Coxeter groups in relation with Conjecture \ref{conjecture_cox}, in Section \ref{examples}. We also discuss Fuchsian groups, and funndamental groups of hyperbolic 3--manifolds, at the end of this section. In section \ref{direct} we give simpler proofs for some particular cases of Theorems \ref{Main_artin} and \ref{Main_coxeter}. The proof of the former is given in Section \ref{sec_RAAG} and that of the latter in Section \ref{sec_RACG}. The appendix \ref{appendix} contains some results on symmetric groups which we use in these proofs, regrouped here for the reader's convenience.

%%%%%%%%%%%%%%%%%%%%%%%%%%%%%%%%%%%%%%%%%%%%%%%%%%%%%%%%%%%%%%%%%%%%%%%%%%%%%%%%

\section{Examples} \label{examples}

\subsection{Subgroup growth gap for right-angled groups}

In general the subgroup growth type can be essentially any growth type (see \cite[Chapter 13]{Lubotzky_Segal}. Restricting to linear groups, a gap in the growth between polynomial and $n^{\log(n)/\log\log(n)}$ appears (loc. cit., Chapter 8). For right-angled Coxeter groups we observe that the gap is much wider. 

\begin{prp} \label{gap}
  Let $\Gamma$ be a right-angled Coxeter or Artin group. We have
  \[
  \liminf_{n \to +\infty}\frac{\log s_n(\Gamma)}{n\log n} = 0
  \]
  if and only if $\Gamma$ is virtually abelian. In particular the subgroup growth is either polynomial or of type $n^n$. 
\end{prp}

For the representation growth $h_n$ it follows that either $h_n(\Gamma) \ge (n!)^{1+\delta}$ for some $\delta > 0$, or $h_n(\Gamma) \le C^n n!$ for some $C$ in the case $\Gamma$ is virtually abelian (this follows easily from the fact that $t_n(\Gamma) \le n^A n!$, and the recursive formula \cite[Lemma 1.1.3]{Lubotzky_Segal}). In the companion paper \cite{virtabelian} we study the finer growth of $h_n$ and $s_n$ for some examples of virtually abelian Coxeter groups. 

\begin{proof}
  We deal with Artin and Coxeter groups separately. For Artin groups we first note that as soon as there are two non-adjacent vertices, i.e. if the graph is not complete, the group surjects onto a nonabelian free group and hence the limit inferior is positive. On the other hand, for a complete graph the group is free abelian. 

  For Coxeter groups if there are three pairwise non-adjacent vertices then the group surjects onto the group $\ZZ/2*\ZZ/2*\ZZ/2$ for which the limit inferior is positive (using \eqref{2-gps}). Similarly, if neither in a pair of neighbouring vertices is adjacent to a third vertex the group surjects onto $\ZZ/2*(\ZZ/2 \times \ZZ/2)$ for which the limit inferior is positive as well (again using \eqref{2-gps}). On the other hand, in any other situation the group is virtually abelian: assume the graph is not complete, and let $v_1, v_2$ be two nonadjacent vertices, then any third vertex must be adjacent to both hence the group is isomorphic to $D_\infty \times \Gamma$ and we can use an induction argument on $\Gamma$.

  The result on subgroup growth follows immediately since abelian groups have polynomial subgroup growth (this is an easy exercise, see \cite[Lemma 1.4.1]{Lubotzky_Segal}) and for a subgroup $\Gamma' \le \Gamma$ we have $s_n(\Gamma) \le Cn^{|\Gamma/\Gamma'|} s_n(\Gamma')$ (see \cite[Corollary 1.2.4]{Lubotzky_Segal}). 
\end{proof}

Note that this argument actually gives a lower bound of $1$ (respectively $1/4$) in the limit inferior for non-virtually abelian right-angled Artin groups (respectively Coxeter groups). 

%%%%%%%%%%%%%%%%%%%%%%%%%%%%%%%%%%%%%%%%%%%%%%%%%%%%%%%%%%%%

\subsection{Constructing graphs satisfying Conjecture \ref{conjecture_cox}}

\subsubsection{A simple construction of many graphs that satisfy Conjecture \ref{conjecture_cox}} \label{simple_general_const}

One simple construction to turn a graph $\mathcal{G}$ into a graph $\mathcal{G}'$ for which Conjecture \ref{conjecture_cox} holds is to attach two leaves to each vertex of $\mathcal{G}$. If $\mathcal{G}$ has $N$ vertices then $\mathcal{G}'$ will have $3N$ vertices. The number of graphs we construct like this has the same rough growth type as the total number of graphs on $3N$ vertices (it follows from P\'olya's enumeration theorem that this is roughly $\exp(c\;N^2)$).

To see that $\mathcal{G}'$ indeed satisfies Conjecture \ref{conjecture_cox}, we argue as follows. We have two surjections
\[\cox(\Lambda_3)^{*N} \twoheadrightarrow  \cox(\mathcal{G}') \twoheadrightarrow (\ZZ/2\ZZ)^{* 2N} ,\]
where $\Lambda_3$ is the line on $3$ vertices. The first comes from removing the edge set of $\mathcal{G}$ from $\mathcal{G}'$ and the second is obtained by removing $\mathcal{G}$ from $\mathcal{G}'$ entirely. It follows from  Theorem \ref{Main_coxeter} that the groups on the left and right have the same factorial subgroup growth rate. Since the leaves we added to $\mathcal{G}$ realise $\gamma(\mathcal{G}')$, $\mathcal{G}'$ indeed satisfies the conclusion of Conjecture \ref{conjecture_cox}.

\subsubsection{Joins}

Let $\mathcal G_1, \mathcal G_2$ be two graphs. Their {\em join} $\mathcal G_1 \star \mathcal G_2$ is the graph on $\mathcal G_1 \cup \mathcal G_2$ where two vertices in a $\mathcal G_i$ are adjacent if they are adjacent in $\mathcal G_i$, and any two pair of vertices in different $\mathcal G_i$s are adjacent to each other. We have
\[
\cox(\mathcal G_1 \star \mathcal G_2) = \cox(\mathcal G_1) \times \cox(\mathcal G_2). 
\]

\begin{prp}
Suppose there exist $\beta_1,\beta_2\in \mathopen]0,\infty\mathclose[$  so that
\[
\lim_{n\to\infty} \frac{\log(s_n(\cox(\mathcal G_i)))}{n\log(n)} = \beta_i,
\]
for $i=1,2$. Then
  \begin{equation} \label{prod_groups}
\lim_{n\to\infty} \frac{\log(s_n(\cox(\mathcal G_1 \star \mathcal G_2)))}{n\log(n)} = \max\{\beta_1,\beta_2\}.
  \end{equation}
Furthermore if either of the two $\mathcal G_i$'s is not complete then
  \begin{equation} \label{join_graph}
  \gamma(\mathcal G_1 \star \mathcal G_2) = \max(\gamma(\mathcal G_1), \gamma(\mathcal G_2)).
  \end{equation}
In particular, if $\mathcal G_1$ and $\mathcal G_2$ both satisfy Conjecture \ref{conjecture_cox} then their join does as well. 
\end{prp}

\begin{proof}
Not every transitive permutation representation $\rho:\Gamma_1\times \Gamma_2 \to \sym_n$ restricts to a transitive permutation representation of $\rho_1:\Gamma_1\to\sym_n$. However, because $\Gamma_2$ can only permute $\rho_1(\Gamma_1)$-orbits of equal size and $\rho$ is transitive, all the orbits of $\rho_1(\Gamma)$ do need to be of equal size. Let us call this orbit size $d$. If we write $\orb(\Gamma_1) = \{1,\ldots,n\} / \rho_1(\Gamma_1)$, we have a map
\[\rho_2(\Gamma_2) \hookrightarrow \sym_d\wr \sym(\orb(\Gamma_1)) \simeq \sym_d\wr \sym_{n/d},\]
where $\sym(\orb(\Gamma_1))$ denotes the group of bijections $\orb(\Gamma_1) \to \orb(\Gamma_1)$ and $\rho_2$ denotes the restriction of $\rho$ to $\Gamma_2$. Using the fact that $\Gamma_1$ is transitive within its orbits, the number of possibilities for the action of $\Gamma_2$ in the $\rho_1(\Gamma_1)$ orbits is very limited (Lemma \ref{lem_centbd2}). As such, we obtain
  \[
  t_n(\Gamma_1 \times \Gamma_2) \le \sum_{d|n} t_d(\Gamma_1) \cdot (d!)^{\frac n d - 1} \cdot t_{n/d}(\Gamma_2) \cdot d^{n/d}.
  \]
By assumption there exists a function $F:\NN\to\RR$ so that for $i=1,2$,
\[\frac{1}{F(n)} \leq s_n(\Gamma_i)/(n!)^{\beta_i} \leq F(n)\]  
 and $\log(F(n)) = o(n\log(n))$.
Filling this into the bound above, we get
   \[
  \log(t_n(\Gamma_1 \times \Gamma_2)) \le \max_{d|n}\left\{\log\left(n^{\beta_2\;n/d} \;d^{n-d+\beta_1\;d - \beta_2\; n/d}\right)\right\} + o(n \log(n)) 
  \]
  and \eqref{prod_groups} follows.

  To prove \eqref{join_graph} we observe that if $\mathcal C$ is a clique collection which contains a vertex in both of the two $\mathcal G_i$'s then it must consist of a single clique, so that $w(\mathcal C) < 1$ and if one of $\mathcal G_i$ has $\gamma \ge 1$ then any clique collection realising $\gamma(\mathcal G_1 \star \mathcal G_2)$ must be contained in one of $\mathcal G_1$, $\mathcal G_2$. 
\end{proof}

%%%%%%%%%%%%%%%%%%%%%%%%%%%%%%

\subsubsection{Graphs with bounded degree} \label{examples_bdd_deg}

The graphs satisfying the hypotheses of Theorem \ref{Main_coxeter} directly are atypical but there are still many of them. For instance, the simplest construction of regular graphs that satisfy the hypotheses gives at least $(cN)!$ graphs (where we can take $c \ge 1/5-\eps$ for any $\eps > 0$) on $N$ vertices (for a particular sequence of $N$s) (it is easy to modify it to get many more similar examples with better density but we only describe a particularly simple case below). 

We proceeds a follows: let $\mathcal T_r$ be the rooted binary tree of height $r$, to which we add a single vertex adjacent only to the root. We have $|\mathcal T_r| = 2^{r+1}$ and we will construct a graph on $N = 2^{r+1} + 2^{r-1}$ vertices by gluing $2^{r-1}$ vertices to the $2^r$ leaves (excluding the leaf connected to the root) of $\mathcal T_r$. Let $M = 2^{r-1} = N/5$, let $l_1, \ldots, l_{2M}$ these leaves and $v_1, \ldots, v_M$ the additional vertices, then there are $(2M)!/(2^M M!)$ possibilities for joining each $v_i$ to two $l_j$s, such that no leaf is adjacent to more than one $v_i$. Any isomorphism between resulting graphs induces an automorphisms of $\mathcal T_r$, of which there are $2^{2M-1}$ (at each non-leaf vertex we choose whether to switch its descendents or not). In the end we thus get $(2N/5)!/((N/5)!2^{3N/5-1})$ pairwise non-isomorphic graphs on $N$ vertices satisfying the conclusion of Conjecture \ref{conjecture_cox}. 

\medskip

The following easy lemma will allow us to construct, in a different way, large families of examples. 

\begin{lem} 
  Let $\mathcal G_1, \mathcal G_2$ be two graphs and $\mathcal C_i$ two clique collections such that $\gamma(\mathcal G_i) = w(\mathcal C_i)$. Assume that $\mathcal G_1, \mathcal G_2$ both satisfy the conclusion of Conjecture \ref{conjecture_cox}. Let $\mathcal G$ be a graph on $\mathcal G_1 \cup \mathcal G_2$ so that
  \begin{itemize}
 \item[1.] $\mathcal G$   induces $\mathcal G_i$ on its vertices and 
 \item[2.] no vertex of $\mathcal C_1$ is adjacent to a vertex of $\mathcal C_2$ in $\mathcal G$. 
\end{itemize}
Then the conjecture holds for $\mathcal G$ as well. 
\end{lem}

\begin{proof}
  We have
  \[
  \gamma(\mathcal G) \le \gamma(\mathcal G_1 \cup \mathcal G_2) =  \gamma(\mathcal G_1) + \gamma(\mathcal G_2)
  \]
  and since $\mathcal C_1 \cup \mathcal C_2$ is a clique collection in $\mathcal G$ it follows that
  \[
  \gamma(\mathcal G) =  \gamma(\mathcal G_1) + \gamma(\mathcal G_2). 
  \]
  Similarly, the fact that
  \[
  \cox(\mathcal G_1) * \cox(\mathcal G_2) \twoheadrightarrow \cox(\mathcal G) \twoheadrightarrow \cox(\mathcal C_1) * \cox(\mathcal C_2)
  \]
  and that the conjecture holds for the $\mathcal G_i$ (as for the $\mathcal C_i$) implies that
  \begin{multline*}
    \gamma(\mathcal G_1) + \gamma(\mathcal G_2) - 1 \ge \limsup_{n \to +\infty} \frac{\log s_n(\Gamma)}{n\log n} \ge \\
    \liminf_{n \to +\infty} \frac{\log s_n(\Gamma)}{n\log n} \ge \gamma(\mathcal G_1) + \gamma(\mathcal G_2) - 1
  \end{multline*}
  which finishes the proof. 
\end{proof}

Now let $\mathcal I$ be a graph such that there exists a clique collection $\mathcal C$ in $\mathcal I$ with $w(\mathcal C) = \gamma(\mathcal I)$ and $\mathcal I \setminus \mathcal C$ contains at least three vertices such that no two of them are adjacent to each other. Assume that $\mathcal I$ satisfies Conjecture \ref{conjecture_cox}; for example we may take $\mathcal I$ to be the cycle graph on six vertices. Then for any trivalent graph $\mathcal H$ we may construct a graph $\mathcal H_{\mathcal I}$ satisfying the conclusion of the conjecture, by gluing copies of $\mathcal I$ along vertices outside $\mathcal C$, according to the pattern prescribed by $\mathcal H$. Indeed, as $\bigcup_{v \in \mathcal H} \mathcal C$ and $\bigcup_{v \in \mathcal H} \mathcal I$ are respectively a subgraph and an induced subgraph of $\mathcal H_{\mathcal I}$ we get sujections
\[
\Asterisk_{v \in \mathcal H} \cox(\mathcal I) = \cox\left(\bigcup_{v \in \mathcal H} \mathcal I\right) \twoheadrightarrow \cox(\mathcal H_i) \twoheadrightarrow \cox\left(\bigcup_{v \in \mathcal H} \mathcal C \right)
\]
and as we have $\gamma\left(\bigcup_{v \in \mathcal H} \mathcal I\right) = \left(\bigcup_{v \in \mathcal H} \mathcal C\right)$ and both graphs satisfy the conjectures if follows that $\mathcal H_{\mathcal I}$ does as well. 

The number of trivalent graphs on $M$ vertices is of order $(M/2)!$ (up to at most exponential factors) by \cite{Bollobas_nb_graphs} So the construction above yields about $(N/12)!$ graphs on $N$ vertices which satisfy Conjecture \ref{conjecture_cox}. 

%%%%%%%%%%%%%%%%%%%%%%%%%%%%%%%%%%%%%%%%%%%%%%%%%%%%%%%%%%%%

\subsection{Hyperbolic Coxeter groups}

By a theorem of G.~Moussong \cite{Moussong} right-angled Coxeter group is word-hyperbolic if and only if it does not contain $\ZZ^2$. If it is right-angled then in terms of its defining graph this means that the latter does not contain an induced square. From Proposition \ref{gap} we see that any non-virtually cyclic hyperbolic Coxeter group has subgroup growth type $n^n$. 

There are thus plenty of word-hyperbolic Coxeter groups. Geometric hyperbolic Coxeter groups (that is, Coxeter groups which act cocompactly on an hyperbolic space) are much harder to construct and in fact they don not exist in high dimensions. For example right-angled Coxeter groups cannot embed as discrete cocompact subgroups of $\mathrm{PO}(n, 1)$ for $n \ge 5$. For cofinite subgroups the best known upper bound is $12$ while there are examples known up to dimension $8$.

A simple example of such a group, for which we know that the conclusion of Conjecture \ref{conjecture_cox} holds, is any right-angled Coxeter group defined by a tree on more than three vertices. The first family of examples given in \ref{examples_bdd_deg} also satisfy Moussong's condition, if we ask that the added vertices do not connect two leaves at distance 2 (as an illustration the first graph in Figure \ref{example_thm} gives an hyperbolic group, but not the second). Thus we have plenty of examples of hyperbolic Coxeter groups for which our conjecture is true. Unfortunately, beyond the $2$-dimensional case Theorem \ref{Main_coxeter} does not apply to these groups. We will give an elementary proof for a 3-dimensional cofinite group below.

%%%%%%%%%%%%%%%%%%%%%%%%%%%%%%%%%%%%%%%%%%%%%%%%%%%%%%%%%%%%

\subsection{Fuchsian groups} \label{Fuchsian_RACG}

We will give three-dimensional examples for the theorem and for the conjecture in the next subsection. For now we explain how our result overlaps with that of Liebeck and Shalev mentioned above. The right-angled Coxeter group $\Gamma = \cox(\mathcal G)$ is Fuchsian if and only if $\mathcal G$ is either a disjoint union of lines, or a cycle with at least five vertices: clearly, both classes of graphs satisfy the hypotheses of Theorem \ref{Main_coxeter}. Computing $\gamma$ in this case is easy: 
\begin{enumerate}
\item \label{calc_facile1} for $\mathcal G = \mathcal L_r$, a line with $r$ vertices we have $\gamma(\mathcal L_r) = (r+1)/4$ ;

\item \label{calc_facile2} for $\mathcal G = \mathcal P_r$ is a cycle on $r$ vertices we have $\gamma(\mathcal P_r) = r/4$. 
\end{enumerate}
In the former case the group can be either cofinite or not and in the latter case it is always cocompact. If
\[
\mathcal G = \mathcal L_{r_1} \cup \cdots \cup \mathcal L_{r_s}
\]
then $\gamma(\mathcal G) = (s + \sum_{i=1}^s r_i)/4$ by \ref{calc_facile1}. On the other hand the volume of an hyperbolic right-angled polygon with $k$ right angles and $l$ ideal vertices equals $(k+l-2)\pi - k\pi/2$, and if it is a fundamental domain for $\cox(\mathcal G)$ then $l = s$ and $k = \sum_{i=1}^s(r_i-1)$. So we see that
\begin{align*}
  \gamma(\mathcal G) - 1 &= \frac 1 4 \left(-4 + s + \sum_{i=1}^s r_i \right) \\
  &= \frac 1 4(-4 + l + (k + l)) \\
  &= \frac{ (-2 + l + k/2)\pi}{2\pi} = -\chi(\Gamma)
\end{align*}
so we recover \eqref{fuchsian} in this case.

The cocompact case is immediate: the volume of a right-angled $r$-gon in $\HH^2$ is $(r-2)\pi - r\pi/2 = (r/4 - 1)2\pi$ so by \ref{calc_facile2} we see that our result and \eqref{fuchsian} are also in accordance for this case.

%%%%%%%%%%%%%%%%%%%%%%%%%%%%%%%%%%%%%%%%%%%%%%%%%%%%%%%%%%%%

\subsection{Hyperbolic three--manifolds and orbifolds}\label{sec_hyp3folds}

We saw that in the case of Fuchsian groups the subgroup growth rate (which we consider here only through the factorial growth rate, as in \eqref{growth_rate}) and the covolume are linearly related. We will see here that this is not the case for lattices in three-dimensional hyperbolic space. There is still some relation, though not as precise and only in one direction: a result of T.~Gelander \cite{Gelander_rank} states that there is a constant $C$ such that, if $\Gamma \subset \mathrm{PSL}_2(\CC)$ is a discrete subgroup of finite covolume (a lattice---the result is proven more generally for all lattices in semisimple Lie groups) then it is generated by at most $C\vol(\Gamma\bs \HH^3)$ elements. The subgroup growth rate is at most the number of gerenators minus one so we see that it is linearly bounded by the volume. 

The solution of Thurston's conjectures on hyperbolic 3--manifolds by I.~Agol \cite{Agol} (following the work of D.~Wise) allows to give an overall picture of subgroup growth for their fundamental groups. Let $\Gamma$ be a lattice in $\mathrm{PSL}_2(\CC)$. The two results of loc. cit. which are of interest to us here are:
\begin{enumerate}
\item \label{large_gp} There exists a finite-index subgroup $\Gamma' \subset \Gamma$ which admits a surjective morphism to a nonabelian free group $\Phi$.

\item \label{vf} There exists a surface $S$ of finite type, a pseudo-Anosov diffeomorphism $f$ of $S$ and a subgroup $\Gamma' \subset \Gamma$ of finite index such that:
  \[
  \Gamma' \cong \langle \pi_1(S), t:  tgt^{-1} = f_*g,\, \forall g \in \pi_1(S)  \rangle.
  \]  
\end{enumerate}
The first statement immediately implies that
\[
\liminf_{n \to +\infty} \frac{\log s_{nd}(\Gamma)}{nd\log(nd)} \ge \frac 1 d
\]
where $d = |\Gamma/\Gamma'|$. Note however that Agol and Wise's arguments do not give a bound on $d$. In addition, taking $\Gamma_r$ to be the preimage in $\Gamma'$ of a subgroup of index $r$ in $\Phi$ we see that $\Gamma_r$ surjects onto a free group of rank at least $r$ and thus
\[
\liminf_{n \to +\infty} \frac{\log s_n(\Gamma_r)}{n\log n} \ge r-1 \ge \delta\vol(\Gamma_r\bs \HH^3)
\]
for some $\delta > 0$  depending only on $\Gamma$.

On the other hand, statement \ref{vf} allow to construct examples where the volume goes to infinity but the growth rate of $s_n$ stays bounded. More precisely, let $S, f, \Gamma'$ be as in the statement and let
\[
\Gamma_r = \langle \pi_1(S), t^r \rangle \subset \Gamma'. 
\]
Then $\Gamma_r$ is of index $r$ in $\Gamma'$ (so that $\vol(\Gamma_r \bs \HH^3)$ goes to infinity) but it is of rank at most $3 - \chi(S)$ and so
\[
\limsup_{n \to +\infty} \frac{\log s_n(\Gamma_r)}{n\log n} \le 2 - \chi(S)
\]
is bounded.

In conclusion, the subgroup growth of fundamental groups of hyperbolic 3--manifolds is a land of contrasts: the growth rate is always positive but it can be as large (linear in the volume) or as small (bounded) as possible. 

%%%%%%%%%%%%%%%%%%%%%%%%%%%%%%

\subsubsection{Right-angled Coxeter polytopes in $\HH^3$}

Let $\mathbb X$ be a constant curvature space. A {\em Coxeter polytope} in $\mathbb X$ is a convex polytope $P$, all of whose dihedral angles (angles between the normals to top-dimensional faces) are of the form $\pi/m$ for some integer $m \ge 2$. It is {\em right-angled} if all of these are $\pi/2$. The subgroup of $\isom(\mathbb X)$ generated by the reflections in the top-dimensional faces of $P$ is then a discrete subgroup of $\isom(\mathbb X)$, of which $P$ is a fundamental domain (this follows from Poincar\'e's polyhedron theorem). Since two reflections whose mirrors are perpendicular to each other commute, the group associated to a right-angled polytope is a right-angled Coxeter group.

Here we will look at groups for which $P$ is a finite-volume polytope in $\HH^3$. Compared to the general case the growth rate for $s_n$ when restricted to these groups is well-behaved with respect to the volume. We observe the following fact. 

\begin{prp} \label{bd_volume}
  There exists $0 < c \le C$ such that for every right-angled Coxeter polytope $P$ in $\mathbb H^3$ of finite volume, with Coxeter graph $\mathcal G$, we have
  \[
  c\vol(P) \le \liminf_{n \to +\infty} \frac{s_n(\cox(\mathcal G))}{n\log n} \le \limsup_{n \to +\infty} \frac{s_n(\cox(\mathcal G))}{n\log n} \le C\vol(P).
  \]
\end{prp}

\begin{proof}
  The upper bound follows from Gelander's much more general result and thus we need only to prove the lower bound. Let $P$ have $N$ vertices. By \cite[Corollary 1]{Atkinson_volume_bounds} there exists $C_0$ such that
  \[
  \vol(P) \le c_0N. 
  \]
  It follows that we may replace $\vol(P)$ in the statement by $N$. Let $F$ be the number of faces and $A$ of edges in $P$, then by Euler's formula we have  $N-A+F = 2$. By Andreev's theorem (see \cite[Theorem 3]{Atkinson_volume_bounds}) all vertices of $P$ are 3- or 4-valent and it follows that we have
  \[
  3N/2 \le A \le 2N
  \]
  hence
  \[
  N/2 + 2 \le F \le N + 2
  \]
  and thus we only have to prove that  
  \[
  c|\mathcal G| \le \liminf_{n \to +\infty} \frac{s_n(\cox(\mathcal G))}{n\log n}. 
  \]
  We first note that the graph $\mathcal G$ has $A$ edges and as above we see that $A \le 3F - 6$ and so if we denote by $d(v)$ the degree of a vertex we have :
  \[
  \sum_{v \in \mathcal G} d(v) = 2A \le 6F - 12
  \]
  and by Markov's inequality it follows that at least $N/7$ vertices have degree at most 6 in $\mathcal G$. Let $\mathcal H$ be the graph induced on the vertices of degree at most $6$ in $\mathcal G$. Note that $\alpha(\mathcal G) \geq \alpha(\mathcal H)$. Now \cite[Theorem I.b]{Rosenfeld}\footnote{In fact, the result is only stated for regular graphs in \cite{Rosenfeld}, but the proof works verbatim in the slightly more general case of graphs with boudned degree.} implies that: 
  \[
  \alpha(\mathcal H) \geq \frac{\card{\mathcal H}}{7} \geq \frac{N}{49},
  \]  
  which finishes the proof since $\gamma(\mathcal G) \ge \alpha(\mathcal G)/2$. 
\end{proof}

%%%%%%%%%%%%%%%%%%%%%%%%%%%%%%

\subsubsection{Ideal octahedron} \label{sec_octahedron}

We can rather easily give good bounds for the subgroup growth rate of the reflection group associated to the ideal right-angled octahedron. The graph $\mathcal G$ for this polytope is the cubical graph 
\begin{figure}[h] \label{ratios}
\centering
\includegraphics[trim=0cm 0cm 0cm 0cm, width=.3\textwidth]{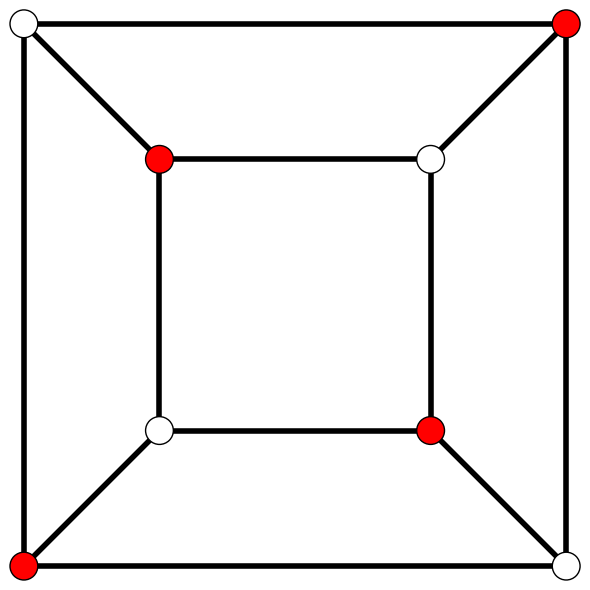}
\end{figure}
and we have $\gamma(\mathcal G) = 2$. On the other hand we can remove four edges from the cube to get a disjoint union of two squares. As the right-angled Coxeter groups associated to a square is $D_\infty \times D_\infty$ (where $D_\infty$ is the infinite dihedral group) we obtain a surjection
\[
(D_\infty \times D_\infty) * (D_\infty \times D_\infty) \to \cox(\mathcal G)
\]
which shows that there exists $C$ such that 
\[
h_n(\cox(\mathcal G)) \le C^{2n} \cdot (n!)^2
\]
as $D_\infty \times D_\infty$ is virtually abelian and hence $h_n(D_\infty \times D_\infty) \le C^nn!$ (see the remark after Proposition \ref{gap}). In particular we obtain
\begin{equation} \label{octahedron}
  \lim_{n \to +\infty} \frac{\log s_n(\Gamma)}{n\log n} = 1.
\end{equation}
It is also easy to compute the covolume of $\cox(\mathcal G)$: the octahedron decomposes as a union of 4 tetrahedra with dihedral angles $((\pi/2)^2, (\pi/4)^4)$ whose volume is thus given by
\[
\vol(O) = 8\Loba(\pi/4) + 4\Loba(\pi/2) = 8\Loba(\pi/4) \cong 3.664...
\]
where $\Loba$ is the Lobachevsky function. 

%%%%%%%%%%%%%%%%%%%%%%%%%%%%%%

%%%%%%%%%%%%%%%%%%%%%%%%%%%%%%%%%%%%%%%%%%%%%%%%%%%%%%%%%%%%%%%%%%%%%%%%%%%%%%%%

\section{Direct combinatorial approaches} \label{direct}

 In this section we record two upper bounds on the growth rates we are after. On the one hand the proofs here are much simpler than those in the following sections. Moreover these proofs give sharper bounds. On the other hand, the methods in this section apply to a much more restricted set of graphs.

\subsection{Right-angled groups associated to bipartite graphs}~

 We start with the case where the defining graph is bipartite.

\subsubsection{Artin groups}~

 It turns out that in the case of right-angled Artin groups, a classical theorem from graph theory, due to K\"onig, gives the growth rate.

\begin{prp} \label{artin_bipartite}
  Let $\mathcal G$ be a bipartite graph and $\Gamma = \art(\mathcal G)$. Then we have
  \[
  s_n(\Gamma) \le \Pi(n)^{\alpha(\mathcal G)} \cdot (n \cdot (n!)^{\alpha(G)-1})
  \]
  for all $n \ge 1$, where $\Pi(n)$ denotes the number of partitions of an integer $n$. 
\end{prp}

It is well-known that $\log\Pi(n) \sim c\sqrt n$ (see for example \cite{erdos_partitions}) so the upper bound above is much sharper than the general one we obtain in the proof of Theorem \ref{Main_artin}.  In particular, it follows from Dixon's theorem \cite{Dixon_generating} that the bound is saturated (as $n\to\infty$) by all RAAGs of the form $\ZZ^2 * \ldots * \ZZ^2$.

\begin{proof}
  We will prove that $h_n(\Gamma) \le \Pi(n)^{\alpha(\mathcal G)} \cdot (n!)^{\alpha(G)}$. The proof is purely graph-theoretical, we will cover $\mathcal G$ by a union of $\alpha(\mathcal G)$ vertices and edges and the use the fact that $h_n(\ZZ^2) = n!\Pi(n)$. 

  For this we use K\"onig's theorem. Before stating it we need some terminology (which won't be used in the rest of the paper). A \emph{matching} in a graph $\mathcal G=(V,E)$ is a set of edges $\mathcal M \subset E$ so that
  \[
  e \cap e' = \emptyset \quad \forall e, e' \in \mathcal M.
  \]
  A \emph{maximal matching} in a graph is a matching with the maximal number of edges among all matching in that graph. Let us write $\mu(\mathcal G)$ for the number of edges in a maximal matching in $\mathcal G$.

  A \emph{vertex cover} of a graph $\mathcal G=(V,E)$ is a set of vertices $\mathcal{C}\subset V$ so that for all $e\in E$ there exists an $v\in\mathcal{C}$ so that $v\in e$. A vertex cover is called \emph{minimal} if it contains a minimal number of vertices among all vertex covers of a graph. We will write $\nu(\mathcal G)$ for the number of vertices in a minimal vertex cover of $\mathcal G$. Because every vertex cover is complementary to an independent set, we have
  \[
  \nu(\mathcal G) + \alpha(\mathcal G) = \card{\mathcal G}.
  \]
  K\"onig's theorem then states that :
  \[
  \mu(\mathcal G) = \nu(\mathcal G).
  \]
  and so we can find a matching of $\card{\mathcal G} - \alpha(\mathcal G)$ edges in $\mathcal G$. These edges account for $2(\card{\mathcal G} - \alpha(\mathcal G))$ vertices of $\mathcal G$. As a result, there are $2\alpha(\mathcal G) - \card{\mathcal G}$ vertices that are not a part of this matching.  So, we obtain a surjection
  \[
  \left(\ZZ^2\right)^{*(\card{\mathcal G}-m(\mathcal G))}*F_{2m(\mathcal G)-\card{\mathcal G}} \to \Gamma.
  \]
  This implies that
  \[
  h_n(\Gamma) \leq \left(\Pi(n)\cdot n! \right)^{\card{\mathcal G} - \alpha(\mathcal G)} (n!)^{2\alpha(\mathcal G)-\card{\mathcal G}} = \left(\Pi(n)\right)^{\card{\mathcal G} - \alpha(\mathcal G)}  \cdot\left(n! \right)^{\alpha(\mathcal G)}.
  \]
  Because $\mathcal G$ is bipartite, we have $\alpha(\mathcal G) \geq \card{\mathcal G}/2$. As such: 
  \[
  h_n(\Gamma)  \leq  \left(\Pi(n) \cdot n! \right)^{\alpha(\mathcal G)}
  \]
  which finishes the proof. 
\end{proof}

%%%%%%%%%%%%%%%%%%%%%%%%%%%%%%

\subsubsection{Coxeter groups}

We will only briefly comment on this case. The proof above adapts {\em verbatim} to Coxeter groups, now using M\"uller's result on the homomorphism growth rate of finite groups \cite{Mul2}. However the bound obtained is not sharp,  because unlike in the case of Artin groups, $h_n((\ZZ/2\ZZ)^2)$ does not have the same factorial growth rate as $h_n(\ZZ/2\ZZ)$. We obtain the following statement.

\begin{prp}
  Let $\mathcal G$ be a bipartite graph and $\Gamma = \cox(\mathcal G)$. Then
  \[
  \limsup_{n\to\infty} \frac{\log(s_n(\Gamma))}{n \log(n)}  \leq (\alpha(\mathcal{G}) + \card{\mathcal G})/4. 
  \]
\end{prp}

Note that the bound above is phrased in terms of the independence number $\alpha(\mathcal G)$ rather than $\gamma(\mathcal G)$, which we expect to be the invariant that determines the limit we are after. Because $\gamma(\mathcal G) \geq \alpha(\mathcal G)/2$, we could rephrase the bound above in terms of $\gamma(\mathcal{G})$. However, this makes the bound strictly weaker: there are many bipartite graphs for which $\gamma(\mathcal G) > \alpha(\mathcal G)/2$.

%%%%%%%%%%%%%%%%%%%%%%%%%%%%%%%%%%%%%%%%%%%%%%%%%%%%%%%%%%%%

\subsection{An elementary proof for complete graphs}\label{sec_cpltegraphs}

Let $r\ge 1$, $\mathcal K_r$ the complete graph on $r$ vertices and $\Gamma = \cox(\mathcal K_r) \cong (\ZZ/2\ZZ)^r$. In this subsection we give a short combinatorial proof of the following fact (which of course follows from M\"uller's much more precise and general result): 
\begin{equation} \label{complete_graphs_rough}
  \lim_{n \to +\infty} \frac{\log h_n(\Gamma)}{n\log n} = 1 - \frac 1{2^r}. 
\end{equation}

For the lower bound we use the left-action of $\Gamma$ on itself, which gives a morphism $\Gamma \to \sym_{2^r}$. By a diagonal embedding, for any $n \ge 2^r$ this gives an embedding
\[
\phi : \Gamma \to (\sym_r)^{\floor{\frac n {2^r}}} \to \sym_n
\]
whose centraliser is
\[
Z_{\sym_n}(\phi(\Gamma)) = \left(\sym_{\floor{\frac n {2^r}}} \wr \Gamma \right) \times \sym_{n - 2^r\floor{\frac n {2^r}}}. 
\]
We see that
\[
\left| Z_{\sym_n}(\phi(\Gamma)) \right| \le  \floor{\frac n {2^r}} ! \cdot 2^{r \cdot n/2^r}\cdot (2^r)!.
\]
Hence it follows that the conjugates of $\phi$ give up to an at most exponential factor $\left( (1-2^{-r})n\right)!$ pairwise distinct representations of $\Gamma$ into $\sym_n$.

Now we prove the upper bound. We recall the notation for sets of involutions :
\begin{gather*}
  \inv_{n, k} = |\{ \sigma \in \sym_n : \sigma^2 = \mathrm{Id}, |\mathrm{fix}(\sigma)| = n - 2k \}| ; \\
  \inv_n = \bigcup_{k=0}^{\floor{n/2}} \inv_{n, k}. 
\end{gather*}
Likewise, we use $\inv(X)$ to denote the involutions on a set $X$. Given $U\subset \sym_n$, we will denote the set of orbits of $\langle U\rangle < \sym_n$ on $\{1,\ldots,n\}$ by $\orb(U)$. We will use the following two lemmas.

\begin{lem}\label{lem_orbits}
  Let $k,n\in\NN$ and $U\subset \sym_n$. Then
  \[
  \card{\st{\pi\in Z_{\inv_n}(U)}{\begin{array}{c} \text{the action of }\pi \text{ on } \orb(U) \\ \text{ has }k\text{ orbits of size }2\end{array} }} \leq 2^n \card{\inv_{\card{\orb(U)},k}}.
  \]
\end{lem}
\begin{proof}
  Write
  \[
  K = \ker\left( Z_{\inv_n}(U) \to \sym(\orb(U)) \right) < \prod_{o\in \orb(U)} Z_{\sym(o)}(U),
  \]
  so that the cardinality we are after has size $\card{K} \cdot \card{\inv_{\card{\orb(U)},k}}$. Because $\langle U\rangle $ acts transitively on each $o\in \orb(U)$, Lemma \ref{lem_transitive} applies, so  
  \[
  \card{K} \leq \prod_{o\in \orb(U)} \card{o} \leq 2^n,
  \]
  from which the lemma follows.
\end{proof}

 The proof of the following we leave to the reader.
\begin{lem}\label{lem_invcount}
  Let $k,n\in\NN$ so that $k\leq n/2$. Then
  \[
  \card{\inv_{n,k}} \leq 2^n\cdot n^k.
  \]
\end{lem}

This will allow us to prove the following upper bound. 

\begin{prp}\label{prp_cpltegraphsmeth3}
  Let $r,n\in \NN$, and $\Gamma = \cox(\mathcal K_r)$we have
  \[
  h_n(\Gamma) \leq 4^{rn}\cdot n^r\cdot n^{\left(1-2^{-r}\right)n}.
  \] 
\end{prp}

\begin{proof} 
Let us once and for all label the vertices of $\mathcal K_r$ by the numbers $1,\ldots,r$. The idea of the proof is to order homomorphisms according to the number of $2$-cycles the $j^{th}$ generator has on the orbits of the first $j-1$ generators. To this end, let us write $\pi_0=e\in\sym_n$ and $\pi_i$ for the image of the generator corresponding to vertex $i$ under our homomorphism. Moreover, we write $k_i$ for the number of $2$-cycles of the image of $\pi_i$ in $\inv(\orb(\pi_0,\ldots,\pi_{i-1}))$. 

The crucial observation is that 
\[\card{\orb(\pi_0,\ldots,\pi_{i})} = \card{\orb(\pi_0,\ldots,\pi_{i-1})} - k_i. \]
Using Lemma \ref{lem_orbits} together with with this observation we obtain
\[h_n(\cox(\mathcal K_r)) \leq 2^{rn} \sum_{k_1=1}^{\floor{n/2}} \card{\inv_{n,k_1}} \sum_{k_2=1}^{\floor{(n-k_1)/2}} \card{\inv_{n-k_1,k_2}} \cdots \sum_{k_r=1}^{\floor{(n-k_1-\ldots-k_{r-1})/2}} \card{\inv_{n-k_1-\ldots -k_{r-1},k_r}}. \]
Now we use Lemma \ref{lem_invcount} together with the fact that $n-k_1-\ldots -k_{j}\leq n$ we get
\[h_n(\cox(\mathcal K_r))  \leq  4^{rn} \sum_{k_1=1}^{\floor{n/2}} n^{k_1} \sum_{k_2=1}^{\floor{(n-k_1)/2}} n^{k_2} \cdots \sum_{k_r=1}^{\floor{(n-k_1-\ldots-k_{r-1})/2}} n^{k_r}. \]
Bounding the innermost sum by its largest term, we get
\begin{eqnarray*}
h_n(\cox(\mathcal K_r))  & \leq &  4^{rn}\cdot n \sum_{k_1=1}^{\floor{n/2}} n^{k_1} \sum_{k_2=1}^{\floor{(n-k_1)/2}} n^{k_2} \cdots \sum_{k_r=1}^{\floor{(n-k_1-\ldots-k_{r-1})/2}} n^{k_{r-1}}  n^{(n-k_1-\ldots-k_{r-1})/2}  \\ 
& = & 4^{rn}\cdot n\cdot  \sum_{k_1=1}^{\floor{n/2}} \sum_{k_2=1}^{\floor{(n-k_1)/2}}  \cdots \sum_{k_r=1}^{\floor{(n-k_1-\ldots-k_{r-1})/2}}  n^{n/2 + (k_1 + \ldots + k_{r-1})/2}.
\end{eqnarray*}
Repeating this estimate another $r-1$ times, we obtain
\[ h_n(\cox(\mathcal K_r))  \leq  4^{rn}\cdot n^r \cdot n^{n\cdot \sum_{j=1}^r 2^{-j}} =  4^{rn}\cdot n^r\cdot n^{\left(1-2^{-r}\right)n}. \]
\end{proof}

%%%%%%%%%%%%%%%%%%%%%%%%%%%%%%%%%%%%%%%%%%%%%%%%%%%%%%%%%%%%

\subsection{Trees}

The same basic idea we used for complete graphs also allows us to get a bound on the number of permutation representations of a right angled Coxeter group associated to a tree.

\begin{prp} \label{trees_directly}
  Let $\mathcal T$ be a finite tree. Then
  \[
  h_n(\cox(\mathcal T)) \leq n^{2\cdot \card{\mathcal T}}\cdot 4^{\card{\mathcal T}\cdot n} \cdot n^{\gamma(\mathcal T)\cdot n}. 
  \]
\end{prp}

\begin{proof}
  Let us root the tree $\mathcal T$ at an arbitrary vertex $v_0$. Our root gives us a way to divide $\mathcal T$ into shells. Denoting the graph distance on $\mathcal T$ by $\dist : \mathcal T \times \mathcal T \to \NN$, we define the $i$-th shell by
  \[
  S_i = \st{v\in \mathcal T}{\dist(v,v_0) =i}
  \]
  for all $i=0,\ldots, R$, where $R$ denotes the largest integer so that $S_R\neq \emptyset$.

  We would like to apply Lemma \ref{lem_orbits} in the same way as we did in the case of complete graphs by going through the shells. However, by only recording the number of two-cycles of the vertices in $S_i$ on orbits of the vertices in $S_{i-1}$, we lose track of how many two-cycles the actual involutions have and hence on their numbers of orbits. This means that given only this data for the involutions corresponding to the shell $S_i$, our bounds on the number of choices for the shell $S_{i+1}$ are not sharp enough.

  The solution is to make a slightly more detailed analysis using the same basic idea. Suppose $v\in S_i$ and $w\in S_{i+1}$ are neighbors. Instead of considering $\sigma_w$ as an involution on $\orb(\sigma_v)$, we will consider it as a pair of involutions in
  \[
  \inv(\orb_1(\sigma_v)) \times \inv(\orb_2(\sigma_v)),
  \]
  where $\orb_j(\sigma_v)$ denotes the number of orbits of $\sigma_v$ on $[n]$ size $j$ for $j\ in \{1,2\}$. A similar argument as in Lemma \ref{lem_orbits} shows that
  \[
  \card{\st{\pi\in Z_{\inv_n}(\sigma_v)}{\begin{array}{c} \text{the action of }\pi \text{ on } \orb_j(\sigma_v) \\ \text{ has }k_j\text{ orbits of size }2,\;j=1,2\end{array} }} \leq 2^n \card{\inv_{\card{\orb_1(\sigma_v)},k_1}}\cdot \card{\inv_{\card{\orb_2(\sigma_v)},k_2}}.
  \]
  In fact, the exact value of the number on the left hand side can also easily be computed, but the above is sufficient for us.

  Note that every vertex $v$ has a unique parent $p(v)$: a neighbor in the shell before its own. In what follows we will label the vertices of $\mathcal T$ by the numbers $1,\ldots, r$ so that the root gets labeled $1$. Moreover we will sort the permutation representations of $\cox(\mathcal T)$ according to the numbers $k_{ij}$ of two-cycles of each vertex $i$ on $\orb_j(\sigma_{p(i)})$ for $i=1,\ldots, r$ and $j=1,2$.

  Using Lemma \ref{lem_invcount} together with the inequality above, we obtain that
  \[
  h_n(\cox(\mathcal T)) \leq 4^{rn} \sum_{(k_{ij})_{i,j} \in \mathcal{K}_n(\mathcal T)} n^{k_{1,1} + k_{1,2} + \ldots+k_{r,1}+k_{r,2}}
  \]
  where
  \[
  \mathcal{K}_n(\mathcal T) = \st{(k_{i,j})_{i,j} \in \Mat_{r \times 2}(\NN) }{
    \begin{array}{c}
      0\leq k_{1,1} \leq \floor{\frac{n}{2}},\; k_{1,2} = 0 \\
      0\leq k_{i,1} \leq \floor{\frac{n-2k_{p(i),1} - 4k_{p(i),2}}{2}},\; i=2,\ldots r \\
      0\leq k_{i,2} \leq \floor{\frac{k_{p(i),1} + 2k_{p(i),2}}{2} },\; i=2,\ldots r
    \end{array}
  }
  \]
  and where we have used the simple observation that if $\sigma_i$ has $k_{i,1}$ two-cycles as an involution in $\inv(\orb_1(\sigma_{p(i)}))$ and $k_{i,2}$ two-cycles as an involution in $\inv(\orb_2(\sigma_{p(i)}))$, then it has $k_{i,1} + 2k_{i,2}$ two cycles as an involution in $\inv_n$.

  Write
  \[
  k_{\max}(n) = \max\st{\sum_{i,j} k_{ij}}{(k_{i,j})_{ij}\in \mathcal{K}_n(\mathcal T)}.
  \]
  Since $\card{\mathcal{K}_n(\mathcal T)} \leq n^{2r}$, we obtain
  \[
  h_n(\cox(\mathcal T)) \leq n^{2r} \cdot 4^{rn} \cdot n^{k_{\max}(n)}.
  \]
  All that remains to show is therefore that
  \[
  k_{\max}(n) \leq \gamma(\mathcal T)\cdot n.
  \]
  To this end, we turn the problem of finding $k_{\max}(n)$ into a convex optimization problem in a real vector space. That is, we define the convex polytope
  \[
  \mathcal{X}_n(\mathcal T) = \st{(x_{i,j})_{i,j} \in \Mat_{r \times 2}(\RR) }{
    \begin{array}{c}
      0\leq x_{1,1} \leq \frac{n}{2},\; x_{1,2} = 0 \\
      0\leq x_{i,1} \leq \frac{n-2x_{p(i),1} - 4x_{p(i),2}}{2},\; i=2,\ldots r \\
      0\leq x_{i,2} \leq \frac{x_{p(i),1} + 2x_{p(i),2}}{2} ,\; i=2,\ldots r
    \end{array}
  }
  \]
  and the number 
  \[
  x_{\max}(n) = \max\st{\sum_{i,j} x_{ij}}{(x_{i,j})_{ij}\in \mathcal{X}_n(\mathcal T)}.
  \]
  Clearly $k_{\max}(n)\leq x_{\max}(n)$. 

  Because we are now maximizing a linear function over a compact convex real polytope, the maximum is realized at a vertex of $\mathcal{X}_n(\mathcal T)$. That is, to find our maximum $x_{\max}(n)$ we need only consider the sequences $(x_{i,j})_{i,j}$ that saturate the inequalities that define $\mathcal{X}_n$.
  
  Eventually, we want to prove that $x_{\max}(n) = \gamma(\mathcal T) \cdot n$.
First we claim that, if a sequence $(x_{ij})_{ij}$ saturates the inequalities, then 
\[ x_{i,1} \in \left\{ \begin{array}{ll} 
\{0, n/2\} & \text{if}\;\; x_{p(i),1} + 2 \; x_{p(i),2} = 0 \\[3mm]
\{0\} & \text{if}\;\; x_{p(i),1} + 2 \; x_{p(i),2} = n/2 
\end{array} \right.
 \]
and  
\[ x_{i,2} \in \left\{ \begin{array}{ll} 
\{0\} & \text{if}\;\; x_{p(i),1} + 2 \; x_{p(i),2} = 0 \\[3mm]
\{0,n/4\} & \text{if}\;\; x_{p(i),1} + 2 \; x_{p(i),2} = n/2 
\end{array} \right.
 \]
Indeed, this follows from induction on the number of vertices of $\mathcal T$. For the tree of one vertex, this follows by definition. The induction step is done using a leaf of $\mathcal T$: if $(x_{i,j})_{0\leq i\leq r,j=1,2}$ is a vertex of $\mathcal X_n(\mathcal T)$, then $(x_{i,j})_{0\leq i\leq r-1,j=1,2}$ is a vertex of $\mathcal X_n(\mathcal T \setminus r)$. So the statement follows by considering the possible values of  $x_{p(r),1} + 2 \; x_{p(r),2}$ and the implications of these for the values of $x_{r,1}$ and $x_{r,2}$.

   This means that vertices $(x_{ij})_{ij}$ of $\mathcal X_n(\mathcal T)$ are determined by the equations
\[ x_{i,1}+x_{i,2} \in \left\{ \begin{array}{ll} 
\{0, n/2\} & \text{if}\;\; x_{p(i),1} + x_{p(i),2} = 0 \\[3mm]
\{0,n/4\} & \text{if}\;\; x_{p(i),1} +  x_{p(i),2} >0
\end{array} \right.
 \]
   and hence that
\begin{equation}\label{eq_optimization}
x_{\max}(n) =  \max\st{\sum_{i} x_i}{\begin{array}{c}
x_1\in\{0,n/2\}, \\[3mm]
 \; x_i \in \left\{ \begin{array}{ll} 
\{0, n/2\} & \text{if}\;\; x_{p(i)}  = 0 \\[3mm]
\{0,n/4\} & \text{if}\;\; x_{p(i)} > 0 
\end{array} \right. 
\end{array}
}.
\end{equation}

   To show that this is equal to $\gamma (\mathcal T)\cdot n$, note that the only complete graphs that can appear as subgraphs of $\mathcal T$ are $K_1$ and $K_2$. As such, we need to show that there exists a maximizer for \eqref{eq_optimization} that is supported on a disjoint union of such subgraphs.
   
   To this end, suppose that $(x_i)_i$ is a sequence that satisfies the conditions of \eqref{eq_optimization} that contains a connected graph on $3$ vertices $a,b$ and $c$ in its support. With respect to the shells, three such vertices can have two types of relations: one parent ($b$) and two children ($a$ and $c$) or a grandparent ($a$), a parent ($b$) and a child ($c$). In both cases, we want to show that the value of $\sum_i x_i$ can be made (not necessarily strictly) larger by choosing a sequence $(x'_i)_i$ that still satisfies the conditions in \eqref{eq_optimization} and satifies $x'_b=0$.
   
   In the first case, we have 
   \[
   x_a + x_b + x_c \in \{n, 3n/4\},
   \]
   depending on whether or not $b$ has a parent and if so whether $x_{p(b)}$ is positive. Setting $x'_b=0$, we are allowed to set $x'_a=x'_c=n/2$. Moreover, if $b$ has other children $\{d_j\}_j$ then $x_{d_j} \in \{0,n/4\}$ and we set 
   \[ 
   x'_{d_j} = 
   \left\{ 
   \begin{array}{ll} 
     0 & \text{if } x_{d_j}= 0 \\[3mm] 
     n/2  & \text{if } x_{d_j} = n/4 
   \end{array} \right.
   \]
   Since none of this does changes any of the conditions on the descendents of $\{d_j\}_j$, $a$ and $c$, we can leave the rest of the sequence as is. Since we now have
   \[
   x'_a+x'_b+x'_c = n,
   \]
   the sum of the resulting sequence has not decreased.
   
  Likewise, in the second case, the sequence we start with satisfies
  \[
  x_a + x_b+x_c \in \{n, 3n/4\} \; \; \text{and} \; \; x_b = x_c = n/4
  \]
  depending on whether or not $a$ has a parent and if so whether $x_{p(a)}$ is positive. Setting $x'_b=0$, we are allowed to set $x'_c=n/2$. Again, if $b$ has other children $\{d_j\}_j$ then $x_{d_j} \in \{0,n/4\}$ and we set 
  \[
  x'_{d_j} = 
  \left\{ 
  \begin{array}{ll} 
    0 & \text{if } x_{d_j}= 0 \\[3mm] 
    n/2  & \text{if } x_{d_j} = n/4 
  \end{array} \right.
  \]
  and again none of the conditions on the descendents of the vertices $c$ and $\{d_j\}_j$ change, which implies that we can leave the rest of the sequence as is. Finally, we again have
  \[
  \sum_i x'_i \geq \sum_i x_i.
  \]

  What the above shows that there exist vertices of $\mathcal{X}_n(\mathcal T)$ that maximize $\sum_{i,j} x_{ij}$ and moreover, when interpreted as functions on $\mathcal T\times \{1,2\}$, are supported on a disjoint union of subgraphs isomorphic to either $K_1$ or $K_2$. For these maximizers, the analysis can be reduced to understanding the maxima of $\mathcal T\simeq K_1$ and $\mathcal T\simeq K_2$. An elementary computation (which is essentially what we did in the proof of Proposition \ref{prp_cpltegraphsmeth3}) shows that these are $n/2$ and $3n/4$ respectively, which means we are done.
\end{proof}

The proof above also illustrates a stark difference with the case of Artin groups: the vertices of $\mathcal{X}_n(\mathcal T)$ that realize the maxima are not exclusively those supported on independent unions of complete subgraphs of $\mathcal T$. A simple example of a tree $\mathcal T$ for which $\mathcal{X}_n(\mathcal T)$ has such vertices is the line $\Lambda_r$ on $r$ vertices. We have
\[
\gamma(\Lambda_r) = \frac{r+1}{4}
\] 
An example of a vertex of $\mathcal{X}_n(\Lambda_r)$ that also has coordinate sum $\frac{r+1}{4}n$ is the vertex
\[
\left(\begin{array}{cccc} n/2 & 0 & \cdots & 0 \\
           0 & n/4 & \cdots & n/4 \end{array}
\right) \in \mathcal{X}_n(\Lambda_r).
\]
Phrased in terms of homomorphisms, this means that the homomorphisms where all the generators have roughly $n/2$ two-cycles contribute enough to show up in the asymptotic we are after, which is very different from the situation of RAAGs.

%%%%%%%%%%%%%%%%%%%%%%%%%%%%%%%%%%%%%%%%%%%%%%%%%%%%%%%%%%%%%%%%%%%%%%%%%%%%%%%%

\section{Sharp rough upper bound for the subgroup growth of RAAGs} \label{sec_RAAG}

Let $\mathcal G$ be a finite graph on $r$ vertices, $\Gamma = \art(\mathcal G)$. The ultimate goal of this section will be to prove:

\begin{prp} \label{upper_bd_RAAG}
  There exists a constant $D>0$ (depending only on $\mathcal{G}$) so that
  \[
  h_n(\art(\mathcal G)) \le D^{n\log\log(n)} (n!)^{\alpha(\mathcal G)}
  \]
  for all $n$ large enough. 
\end{prp}

We will prove this proposition by induction on the number of vertices $r$. That is, we assume that we know that for any graph $\mathcal H$ on at most $r-1$ vertices (so for any proper subgraph of $\mathcal G$) we have
\begin{equation} \label{eq_induction_hyp}
  h_n(\art(\mathcal H)) \le B^{n\log\log(n)} (n!)^{\alpha(\mathcal H)}. 
\end{equation}

Throughout this section we also fix a vertex $v_0 \in \mathcal G$. Let $v_1, \ldots, v_d$ be its neighbours in $\mathcal G$ and $\sigma_0, \ldots, \sigma_d$ the corresponding generators. We enumerate (arbitrarily) the remaining generators of $\Gamma$ as $\sigma_{d+1}, \ldots, \sigma_r$. Finally, we will use the shortened notation $\Gamma$ for $\art(\mathcal G)$. 

%%%%%%%%%%%%%%%%%%%%%%%%%%%%%%%%%%%%%%%%%%%%%%%%%%%%%%%%%%%%

\subsection{Splitting into small and large orbits}~\label{sec_raagsmalllargeorbits}

We will bound $h_n(\Gamma)$ from above by splitting it into more manageable summands as follows : for $\ell, K \ge 0$ let  
\[
L(\ell, K) = \{ \rho \in \Hom(\Gamma, \sym_\ell) : \text{all orbits of } \langle\rho(\sigma_1), \ldots, \rho(\sigma_d) \rangle \text{ are of size } > K \}
\]
and
\[
S(\ell, K) = \{ \rho \in \Hom(\Gamma, \sym_\ell) : \text{all orbits of } \langle\rho(\sigma_1), \ldots, \rho(\sigma_d) \rangle \text{ are of size } \le K \}. 
\]

\begin{lem} \label{dec_small_large_lem}
There is $C_0 > 0$ (depending only on $\mathcal G$) such that for all $n \ge 1$ we have : 
\begin{equation} \label{dec_small_large}
  h_n(\Gamma) \le C_0^n \sum_{m=0}^n \binom{n}{m} \card{L(n-m, K)} \cdot \card{S(m, K)}. 
\end{equation}
\end{lem}

\begin{proof}
  Given a set of vertices $W\subset V(\mathcal{G})$ and a $\rho\in\Hom(\Gamma,\sym_n)$, write 
  \[
  \mathcal{O}^\rho_{\geq K}(W)\subset \{1,\ldots,n\}
  \]
  for the union of all orbits of size at least $K$ of $\rho(\st{\sigma_w}{w\in W})$. In what follows, we will abuse notation slightly and write $\rho(W) = \rho(\st{\sigma_w}{w\in W})$. 

  Given $m, n \in \NN$ with $m \leq n$ and two disjoint sets of vertices $W_0,W_1\subset V(\mathcal G)$, define
  \[
  H_{n,m}(\Gamma,W_0,W_1) = \st{\rho\in\Hom(\Gamma,\sym_n)}{\substack{\displaystyle{\mathcal{O}^\rho_{\geq K}(W_0) = \{1,\ldots, n-m\}\;\text{and}}\\[2mm] \displaystyle{\rho(W_1) \;\text{preserves}\; \{1,\ldots,n-m\}}}} ; 
  \]
  note that $H_{n, m}(\Gamma, N_1(v_0), \mathcal G \setminus N_1(v_0))$ is naturally identified with $S(m, K) \times L(n-m, K)$. Finally, we write
  \[
  h_{n,m}(\Gamma,W_0,W_1) =\card{H_{n,m}(\Gamma,W_0,W_1)} 
  \]
  so that
  \begin{equation} \label{xhfgsyktfs}
  h_{n, m}(\Gamma, N_1(v_0), \mathcal G \setminus N_1(v_0)) = |S(m, K)| \cdot |L(n-m, K)|. 
  \end{equation}

  Given $v\in V(\mathcal G)\setminus (W_0\cup W_1)$, we have
  \[
  h_{n,m}(\Gamma,W_0,W_1\cup\{v\}) = \sum_{\rho\in  H_{n,m}(\art(\mathcal{G}\setminus\{v\}) ,W_0,W_1)} \card{\cent_{\sym_{n-m}\times\sym_m}(\rho(N(v))},
  \]
  where we have written $\sym_{n-m}\times\sym_m = \sym(\{1,\ldots, n-m\})\times\sym(\{n-m+1,\ldots,n\})<\sym_n$. Now we use Lemma \ref{lem_centbd2} and obtain
  \begin{eqnarray*}
  h_{n,m}(\Gamma,W_0,W_1\cup\{v\}) & \geq & 2^{-n} \sum_{\rho\in  H_{n,m}(\art(\mathcal{G}\setminus\{v\}) ,W_0,W_1)} \card{\cent_{\sym_n}(\rho(N(v))}  \\
  & = & 2^{-n}\cdot h_{n,m}(\Gamma,W_0,W_1).
  \end{eqnarray*}
  Applying this to all vertices in $\mathcal G \setminus N_1(v_0)$ in turn we see that
  \begin{eqnarray*}
  h_n(\Gamma) & = & \sum_{m=0}^n \binom{n}{m} h_{n,m}(\Gamma,N_1(v_0),\emptyset) \\
  & \leq & 2^{\card{V(\mathcal G)}-\card{N_1(v_0)}}\sum_{m=0}^n \binom{n}{m} h_{n,m}(\Gamma,N_1(v_0),V(\mathcal G)\setminus N_1(v_0)).
  \end{eqnarray*}
  which together with \eqref{xhfgsyktfs} finishes the proof. 
\end{proof}

%%%%%%%%%%%%%%%%%%%%%%%%%%%%%%%%%%%%%%%%%%%%%%%%%%%%%%%%%%%%

\subsection{Large orbits}~

 Now all that remains is to control the quantities $\card{L(n-m, K)}$ and $\card{S(m, K)}$. We start with the former:

\begin{lem} \label{large_lem}
  Let $K = \floor{\log(n)}$. There exists a constant $C_1>0$ (depending only on $\mathcal{G}$), so that for all $m,n\in\NN$ so that $m\leq n$ we have:
  \begin{equation} \label{large}
  \card{L(n-m, K)} \le C_1^{n\log\log(n)} \left( (n-m)! \right)^{\alpha(\mathcal G)}
  \end{equation}
\end{lem}

\begin{proof}
  Let
  \[
  \Gamma' = \langle \sigma_1, \ldots \sigma_r \rangle \cong \art(\mathcal G \setminus v_0).
  \]
  We have for any integer $k \ge 1$ that : 
  \[
  \card{L(k, K)} \le h_k(\Gamma') \cdot \max_{\rho\in L(k,K)} \card{\cent_{\sym_k}(\rho(\sigma_1), \ldots, \rho(\sigma_d))}. 
  \]
  Let $\rho\in L(k,K)$ and $H = \cent(\rho(\sigma_1), \ldots, \rho(\sigma_d))$. Moreover let $H'$ be the normal subgroup of $H$ which preserves setwise every orbit of $\langle \rho(\sigma_1), \ldots, \rho(\sigma_d) \rangle$. Let $\orb=\orb(\rho(\sigma_1), \ldots, \rho(\sigma_d))$ be the set of orbits, so that $H/H' \hookrightarrow \sym(\orb)$. Since all orbits have cardinality at least $K$ we have $|\orb| \le k/K$ of them and it follows that $|H/H'| \le |\sym(\orb)| \le \floor{k/K}!$. Note that if $k \le n$ and $K = \floor{\log n}$ we have
  \[
  \log(\floor{k/K}!) \le k/K \log(k/K) \le \frac{n}{\log n} \cdot (\log(n) - \log\log n) \le n 
  \]
  hence we have $\floor{k/K}! \le 3^n$ and it follows that $|H/H'| \le 3^n$. 

  On the other hand $H'$ is the product of centralisers of the restrictions to the orbits. Letting $\lambda_1, \ldots, \lambda_{|\orb|}$ be the cardinalities of the orbits, we get that $H'$ is of size at most $\prod_i 2^{\lambda_i}=2^k$ by Lemma \ref{lem_transitive}. 

  Putting the above together for $k= n-m$, we get that for $1 \le m \le n$ and $K=\floor{\log n}$ we have :
  \[
  |L(n-m, K)| \le \max_\rho (|H/H'| \cdot |H'|) \cdot h_{n-m}(\Gamma') \le 6^n \cdot h_{n-m}(\Gamma')
  \]
  for all $n \geq 0$. Since $\alpha(\mathcal G \setminus v_0) \le \alpha(\mathcal G)$ it follows from the induction hypothesis \eqref{eq_induction_hyp} that for some $C_0$ depending only on $\mathcal G$ we have
  \[
  h_{n-m}(\Gamma') = h_{n-m}(\art(\mathcal G \setminus v_0)) \le C_0^{n\log\log n}((n-m)!)^{\alpha(\mathcal G)}
  \]
  and so we get 
  \[
  |L(n-m, K)| \le C_1^{(n-m)\log\log(n-m)} \left( (n-m)! \right)^{\alpha(\mathcal G)},
  \]
  for some constant $C_1>0$. 
\end{proof}

%%%%%%%%%%%%%%%%%%%%%%%%%%%%%%%%%%%%%%%%%%%%%%%%%%%%%%%%%%%%

\subsection{Small orbits}~\label{sec_raagsmall}

For $\card{S(m, K)} $ we prove:

\begin{lem} \label{small_orbits}
  There exists a constant $C_2>0$ (depending only on $\mathcal{G}$), so that for all $m,n\in\NN$ so that $m\leq n$ we have:
  \begin{equation} \label{small}
  \card{S(m, K)} \le C_2^{n\log\log(n)} (m!)^{\alpha(\mathcal G)}
  \end{equation}
\end{lem}

\begin{proof}
  Let $\rho \in S(m, K)$ and let $\pi$ be the partition of $[1, m]$ into the orbits of the group $\rho(\langle \sigma_1, \ldots, \sigma_d \rangle)$, the elements of which are of size at most $K$. Furthermore, let $z_\pi$ the size of the corresponding centraliser: if $\pi = (1^{\lambda_1} \cdots K^{\lambda_K})$ we have
  \[
  z_\pi = \prod_{i=1}^K i^{\lambda_i} \cdot \lambda_i!.
  \]
  In particular, by Lemma \ref{lem_transitive} we have that
  \begin{equation} \label{simpl_centra}
    \card{Z(\sigma_1, \ldots, \sigma_d) } \le \prod_i (2^i)^{\lambda_i} \lambda_i! = 2^{\sum_i i\lambda_i} \prod_i  \lambda_i! \leq 2^m z_\pi.
  \end{equation}

  Now let
  \[
  \Gamma'' = \langle \sigma_{d+1}, \ldots, \sigma_r \rangle \cong \art(\mathcal G \setminus N(v_0)). 
  \]
  For a $\rho \in S(m, K)$ we have at most $h_m(\Gamma'')$ possibilities for $(\rho(\sigma_{d+1}), \ldots, \rho(\sigma_r))$. The number of possibilities for $(\rho(\sigma_1, \ldots, \rho(\sigma_d))$ with the orbit profile $\pi = (1^{\lambda_1} \cdots K^{\lambda_K})$ is at most :
  \begin{equation} \label{est_nb_13}
    \frac{m!}{z_\pi} \prod_{i=1}^K (i!)^{d\lambda_i}
  \end{equation}
  Here the first factor counts the number of realisations of $\pi$ in $[1,m]$ and the second is the number of representations of the free group on $d$ generators preserving each element of the partition. We have:
  \[
  \log \left( \prod_{i=1}^K (i!)^{d\lambda_i} \right) \le \sum_{i=1}^K d\lambda_i \cdot (i \log i) \le d\log(K) \sum_i \lambda_i \cdot i = d\log\log(n) \cdot m
  \]
  so that
  \begin{equation} \label{LABEL}
    \prod_{i=1}^K (i!)^{\lambda_i} \le 3^{dm\log\log(n)}
  \end{equation}

  In the end, using \eqref{simpl_centra}, \eqref{est_nb_13} and \eqref{LABEL} we obtain : 
  \begin{equation} \label{small_beta}
    \card{S(m, K)} \le h_m(\Gamma'') \sum_\pi 2^m z_\pi \cdot \frac{m!}{z_\pi} \prod_{i=1}^K (i!)^{\lambda_i} \le 6^{m\log\log(n)} \cdot m! \cdot h_m(\Gamma'')
  \end{equation}
for large $n$. Now we have the following easy lemma.

  \begin{lem} \label{gap1_indep_nb}
  For any vertex $v$ of $\mathcal G$, if $\mathcal G''$ is the subgraph induced on vertices at distance at least 2 of $v$ then $\alpha(\mathcal G'') \le \alpha(\mathcal G) - 1$. 
  \end{lem}

  So we get from  the induction hypothesis \eqref{eq_induction_hyp} that
  \[
  h_m(\Gamma') \le B^{m\log\log m} (m!)^{\alpha(\mathcal G \setminus N(v_0))} \le B^{m\log\log m} (m!)^{\alpha(\mathcal G) - 1}
  \]
  and together with \eqref{small_beta} this finally implies that 
  \[
  \card{S(m, K)} \le 6^{m\log\log(n)} B^{m\log\log(n)} \cdot (m!)^{\alpha(\mathcal G'')} \cdot m! \le C_2^{n\log\log(n)} (m!)^{\alpha(\mathcal G)},
  \]
  for some constant $C_2>0$. 
\end{proof}

%%%%%%%%%%%%%%%%%%%%%%%%%%%%%%%%%%%%%%%%%%%%%%%%%%%%%%%%%%%%

\subsection{Conclusion}

We are now ready to prove our bound. Plugging \eqref{large} and \eqref{small} into \eqref{dec_small_large} we get :
\begin{align*}
  h_n(\Gamma) &\le  C_0^n \cdot (C_1\cdot C_2)^{n\log\log(n)} \sum_{m=0}^n (m!)^{\alpha(\mathcal G)} \left((n-m)!\right)^{\alpha(\mathcal G)} \\
  & \le C_0^n \cdot (C_1\cdot C_2)^{n\log\log(n)} \cdot n \cdot (n!)^{\alpha(\mathcal G)}
\end{align*}
so we can conclude that there exists $D$ depending only on $\mathcal G$ such that
\[
  h_n(\Gamma) \le D^{n\log\log(n)} (n!)^{\alpha(\mathcal G)}. 
\]

%%%%%%%%%%%%%%%%%%%%%%%%%%%%%%%%%%%%%%%%%%%%%%%%%%%%%%%%%%%%%%%%%%%%%%%%%%%%%%%%

\section{Sharp rough upper bound for the subgroup growth of some RACGs} \label{sec_RACG}

In this section we prove the upper bound in Theorem \ref{Main_coxeter}. We will use an induction argument based on the following more precise statement. We denote by $\mathcal P_{\le l}(n)$ the set of partitions $\pi$ of $\{1, \ldots, n\}$ with all parts at most $l$, and for such $\pi$ we denote by $\pi_j$ the number of its blocks which are of size $j$. If $\mathcal G$ is a graph, $S$ a set of vertices of $\mathcal G$ and $\underline\pi \in \mathcal P(n)^S$ we define
\[
h_n^{\underline\pi}(\cox(\mathcal G)) = |\{ \rho \in \hom(\cox(\mathcal G), \sym_n):\: \forall v \in S,\, \rho(\sigma_v) \in \stab(\pi_v)\}|.
\]

\begin{prp} \label{coxeter_constraints}
  Let $\mathcal G$ be a graph satisfying the hypotheses of Theorem \ref{Main_coxeter} and $S$ a set of leaves of $\mathcal G_0$. Let $\underline\pi \in \mathcal P_{\le 2}(n)^S$. Then
  \[
  h_n^{\underline\pi}(\cox(\mathcal G)) \le C^{n\log\log(n)} \frac{(n!)^{\gamma(\mathcal G)}} {(n!)^{\frac 1{2n} \sum_{v \in S} \pi_{v, 2}}}.
  \]
\end{prp}

The proof of this takes the whole section. We proceed by steps, first dealing with the simplest cases (graphs with one, two or three vertices in Lemmas \ref{1vert}, \ref{2vert} and \ref{3vert}), then deducing the proposition for all trees (Proposition \ref{tree_constraints}). From the case of trees it is not hard to deduce the case where $\mathcal G$ is a tree with isolated vertices glued to the leaves (Proposition \ref{gluvert}). Finally (in subsection \ref{glustar}), we prove the full version using an argument similar to what we did with RAAGs. 

In all proofs below $C$ is a large enough constant which we allow to vary between instances but in the end will be independent of all parameters except the graph $\mathcal G$ (we use this somewhat unusual convention instead of the big O notation because it seems typographically more adapted to our proof). 

%%%%%%%%%%%%%%%%%%%%%%%%%%%%%%%%%%%%%%%%%%%%%%%%%%%%%%%%%%%%

\subsection{Basic cases}

\begin{lem} \label{1vert}
  Let $0 \le k \le m$ and $\pi \in \mathcal P_{\le 2}(m)$. Let $N_{k, \pi}^1(m)$ denote the number of involutions in $\sym_m$ which preserve $\pi$ and have $k$ fixed points. Then
  \[
  N_{k, \pi}^1(m) \le C^m (m!)^{\frac 1 2 - \frac{\pi_2}{2m} - \frac k{4m}}.
  \]
\end{lem}

\begin{proof}
   Let us order the involutions $\sigma$ that preserve $\pi$ by the numbers $k_1$ and $k_2$ of $1$-blocks and $2$-blocks respectively of $\pi$ that contained in the fixed point set of $\sigma$. So $k=2k_2+k_1$. Note that the number of choices for the decomposition of the parts of $\pi$ into fixed blocks and non-fixed blocks is exponential in $m$.
  
   The involution $\sigma$ preserves the unions of $1$- and $2$-blocks respectively. The number of possibilities for the action of $\sigma$ on the union of $1$-blocks is at most
  \[
  C^m\left( \frac{\pi_1 - k_1} 2 \right)! = C^m \left( \frac{m - k_1} 2 - \pi_2 \right)!.
  \]
  Indeed, this is just the number of involutions on $\pi_1$ points fixing $k_1$ of them (see Lemma \ref{lem_invcount}). The number of choices for the action on the union of $2$-blocks is at most
  \[
  C^m\left( \frac{\pi_2 - k_2} 2 \right)! .
  \]
  Multiplying these two inequalities we get that the number of choices for $\sigma$ is at most
  \begin{align*}
  C^m\left( \frac m 2 - \frac{\pi_2} 2 - \frac{k_1 + k_2} 2 \right)! &= C^m \left( \frac m 2 - \frac{\pi_2} 2 - \frac{k - k_2} 2 \right)! \\
  & \le C^m \left( \frac m 2 - \frac{\pi_2} 2 - \frac k 4 \right)!,
  \end{align*}
 which, together with Stirling's approximation, finishes the proof. 
\end{proof}

\begin{lem} \label{2vert}
  Let $m \ge 1$ and $\pi \in \mathcal P_{\le 2}(m)$. Let $N_\pi^2(m)$ denote the number of pairs of commuting involutions in $\sym_m$, such that the first one preserves $\pi$. Then
  \[
  N_\pi^2(m) \le C^m (m!)^{\frac 3 4 - \frac{\pi_2}{2m}}. 
  \]
\end{lem}

\begin{proof} If an involution $\sigma$ has $k$ fixed points, then
\[ \card{Z_{\inv_m}(\sigma)} = \card{\inv_k} \cdot \card{\inv_{m-k}} \leq C^m \; k^{k/2} \; (n-k)^{(n-k)/2} \]
(see Lemma \ref{lem_invcount}). Combining this with Lemma \ref{1vert}, we obtain
\[N_\pi^2(m) = \sum_{k=0}^{\floor{m/2}} N_\pi^1(m)\cdot \card{Z_{\inv_m}(\sigma)} \leq C^m (m!)^{\frac 3 4 - \frac{\pi_2}{2m}}.\]
\end{proof}

The last case we need for the induction is harder, and we will need to use results on Fuchsian groups from \cite{Liebeck_Shalev} to deduce it. 

\begin{lem} \label{3vert}
  Let $\mathcal G_0$ be the graph on three vertices $v_1, v_2, v_3$ where $v_3$ is adjacent to both $v_1$ and $v_2$. Let $m \ge 1$ and $\underline\pi = (\pi_{v_1}, \pi_{v_2}) \in \mathcal P_{\le 2}(m)^2$. Let
  \[
  N_{\underline\pi}^3(m) = h_m^{\underline\pi}(\cox(\mathcal G_0)).
  \]
  Then
  \[
  N_{\underline\pi}^3(m) \le C^{m\log\log m} (m!)^{1 - \frac{\pi_{1,2} + \pi_{2,2}}{2m}}.
  \]
\end{lem}

\begin{proof}
  Let $\lambda = \pi_1 \vee \pi_2$ (the set partition generated by $\pi_1$ and $\pi_2$). Then $N_{\underline\pi}^3(m)$ depends only on $\pi_{1,2}$, $\pi_{2,2}$ and the integer partition associated to $\lambda$, as any two $\pi_1, \pi_2$ that generate the same partition $\lambda$ are conjugated to each other by an element of $\sym_m$.

  We will first reduce to the case where all blocks of $\lambda$ have the same size $\le \log(m)$. Using the proof of the estimate \eqref{dec_small_large} (which also works when $\Gamma$ is a Coxeter group, using the second statement in Lemma \ref{lem_centbd2}) we obtain that
  \[
  N_{\underline\pi}^3(m) \le C^m \sum_{k=0}^m N_{\underline\pi_{\ge \log(m)}}^3(m-k) N_{\underline\pi_{\le \log(m)}}^3(k)
  \]
  where $\pi_{\le l}$ denotes the partition $\pi$ induces on the union of blocks of $\lambda$ of size at most $l$.

  \medskip

  We first estimate the factor $N_{\underline\pi_{\ge \log(m)}}^3(m-k)$. We will do this by comparing $N_{\underline\pi_{\ge \log(m)}}^3(m-k)$ to the homomorphism count of a larger group: the RACG defined by the line on five vertices.
  
  Note that $(\pi_1)_{\ge \log(m)}$ and $(\pi_2)_{\ge \log(m)}$ define involutions in $\sym_{m-k}$. We will denote these involutions by $\sigma_{\pi_1},\sigma_{\pi_2}\in\sym_{m-k}$ respectively. Let $H$ be the centraliser in $\sym_{m-k}$ of $\langle \sigma_{\pi_1}, \sigma_{\pi_2} \rangle$. Then, since $\lambda_{\ge \log(m)}$ is the partition of $\{1,\ldots,m-k\}$ into orbits of $\langle \sigma_{\pi_1}, \sigma_{\pi_2} \rangle$ we have that 
  \[
  |H| \le  \prod_{j=\log(m)}^m (\lambda_j)! \cdot j \le C^{m\log\log m},
  \]
  where we have used Lemma \ref{lem_transitive}. Thus 
  \[\card{\sym_{\log(m)}^{\lambda_{\log(m)}} \times \ldots \times \sym_m^{\lambda_m}} \geq C^{-m\log\log m} (m-k)!\]  
  and hence, using Proposition \ref{nb_involutions}, there are at least $C^{-m\log\log m}(m-k)!$ pairs of involutions $\sigma_{\pi_1} = \sigma_4, \sigma_{\pi_2} = \sigma_5$ that generate a partition with profile $\lambda$. 
  
  On the other hand, the involutions $\sigma_i$, $1 \le i \le 5$ satisfy $\sigma_4\sigma_1 = \sigma_1\sigma_4$ and $\sigma_5\sigma_2 = \sigma_2\sigma_5$. So we have that
  \[
  |\{\sigma_4, \sigma_5 \text{ as above}\}| \cdot N_{\underline\pi_{\ge \log(m)}}^3(m-k)  \le h_{m-k}(\mathcal L_5)
  \]
  where $\mathcal L_5$ is the line on five vertices (numbered 4,1,3,2,5 in order). We thus know by Proposition \ref{trees_directly} that
  \[
  (m-k)! N_{\underline\pi_{\ge \log(m)}}^3(m-k) \le C^{m\log\log m} ((m-k)!)^{\frac 3 2}
  \]
  and hence
  \begin{equation} \label{large_part_3vert}
    N_{\underline\pi_{\ge \log(m)}}^3(m-k) \le C^{m\log\log m} ((m-k)!)^{\frac 1 2}.
  \end{equation}

  \medskip

  Now we deal with $N_{\underline\pi_{\le \log(m)}}^3(k)$. Using the same separating trick $\log\log(m)$ times (with a ``divide and conquer'' approach) we get that
  \[
  N_{\underline\pi_{\le \log(m)}}^3(k) \le C^{m\log\log(m)} \sum_{k_1+\cdots+k_{\log m}=k} \prod_{l=1}^{\log m} N_{\underline\pi_{=l}}^3(k_i). 
  \]
  To estimate $N_{\underline\pi_{=l}}^3(k)$ we use the same method as above, but with a non-right angled Coxeter group to which we apply the results of \cite{Liebeck_Shalev}. As above, we see that the number of $\sigma_4, \sigma_5$ is at least
  \[
  C^{-m}(k!) / (k!)^{1/l}. 
  \]
  
  Again set $\sigma_4=\sigma_{\pi_1}$ and $\sigma_5=\sigma_{\pi_2}$. 
  
  We now have to subdivide the problem further. We first consider the case where each block of $\lambda$ contains two 1-blocks of $\pi_1$ or $\pi_2$ In this case $\sigma_4, \sigma_5$ satisfy the relation $(\sigma_4\sigma_5)^l = 1$, and it follows that
  \[
  |\{\sigma_4, \sigma_5 \text{ as above}\}| \cdot N_{\underline\pi_{=l}}^3(k) \le h_k(\Gamma_l)
  \]
  where $\Gamma_l$ is the Coxeter group associated to the pentagon with one edge marked by $l$. By Liebeck and Shalev's asymptote \eqref{fuchsian} we thus get that
  \[
  (k!)^{1 - \frac 1 l} \cdot N_{\underline\pi_{=l}}^3(k) \le C^k (k!)^{\frac 3 2 - \frac 1{2l}}
  \]
  and we finally obtain that in this case:
  \begin{equation} \label{small_1_3vert}
    N_{\underline\pi_{=l}}^3(k) \le C^k (k!)^{\frac 1 2 + \frac 1 {2l}}.
  \end{equation}
  This is the result we wanted since we have $(\pi_{1, =l})_1 + (\pi_{2, =l})_1 = 2k/l$. 

  Finally, we deal with the case where each block of $\lambda$ contains only 2-blocks of either $\pi_1$ or $\pi_2$. In this case $l$ is necessarily even and $\sigma_4, \sigma_5$ must satisfy the relation $(\sigma_4\sigma_5)^{l/2} = 1$. Using the same approach as above we get that
  \[
  (k!)^{1 - \frac 1 l} \cdot N_{\underline\pi_{=l}}^3(k) \le C^k (k!)^{\frac 3 2 - \frac 1 l}
  \]
  and it follows that
  \begin{equation} \label{small_2_3vert}
    N_{\underline\pi_{=l}}^3(k) \le C^k (k!)^{\frac 1 2}
  \end{equation}
  in this case, which finishes the proof (note that $(\pi_{1, =l})_1 + (\pi_{2, =l})_1 = 0$ in this case). 
\end{proof}

%%%%%%%%%%%%%%%%%%%%%%%%%%%%%%%%%%%%%%%%%%%%%%%%%%%%%%%%%%%%

\subsection{Counting representations with constraints for trees}

\begin{prp} \label{tree_constraints}
  Let $\mathcal G_0$ be a tree and $S$ a set of non-adjacent leaves of $\mathcal G_0$. Let $\underline\pi \in \mathcal P_{\le 2}(n)^S$. Then
  \begin{equation} \label{eqn_tree_constraints}
  h_n^{\underline\pi}(\cox(\mathcal G_0)) \le C^{n\log\log(n)} \frac{(n!)^{\gamma(\mathcal G)}} {(n!)^{\frac 1{2n} \sum_{v \in S} \pi_{v, 2}}}.
  \end{equation}
\end{prp}

\begin{proof}
  We prove this by induction on the number of vertices. When there are only one, two or three vertices we may just apply Lemma \ref{1vert} (by summing over $k$), Lemma \ref{2vert} or Lemma \ref{3vert}.

  \medskip

  We then start by proving \eqref{eqn_tree_constraints} for stars with at least three leaves. For this we note that if $\mathcal G_0$ is a star graph, $S = \{v_1, \ldots, v_d\}$ its set of leaves and $\underline\pi \in \mathcal P_{\le 2}(n)^S$ we have 
  \[
  h_n^{\underline\pi}(\cox(\mathcal G_0)) \le N_{(\pi_{v_1}, \pi_{v_2})}^3(n) \cdot \prod_{j=3}^d N_{\pi_{v_j}}^1(n)
  \]
  and estimating the right-hand side using Lemmas \ref{1vert} and \ref{3vert} gives the result in this case. 

  \medskip

  Now assume that $\mathcal G_0$ has at least four vertices. Let $v$ be a leaf and $w$ its unique neighbour. We may choose $v$ so that we are in one of the two following situations:
  \begin{enumerate}
  \item \label{more_leaves} $w$ is adjacent to at least two leaves ;

  \item \label{one_leaf} $v$ is the only leaf adjacent to $w$ and the latter has valency 2.
  \end{enumerate}
  (To prove this just observe that if this is not the case for $v$, then taking another leaf $v'$ at maximal distance from $v$ one of the two is satisfied by $v'$).

  Assume first that we are in the situation where \eqref{more_leaves} holds. Let $v=v_1, \ldots, v_d$ be the leaves of $\mathcal G_0$ adjacent to $w$. Let
  \[
  \mathcal G_1 = \mathcal G_0 \setminus \{w, v_1, \ldots, v_d\}, \, S_1 = S \cap \mathcal G_1
  \]
  and $\mathcal G_2$ the graph induced on $\{w, v_1, \ldots, v_d\}$ with $S_2 = \{v_1, \ldots, v_d\}$. Let $\underline\pi^1 \in \mathcal P_{\le 2}(n)^{S_1}$ and $\underline\pi^2 \in \mathcal P_{\le 2}(n)^{S_2}$ such that $\underline\pi = (\underline\pi^1, \underline\pi^2)$. We have
  \begin{align*}
  h_n^{\underline\pi}(\cox(\mathcal G_0)) &\le h_n^{\underline\pi^1}(\cox(\mathcal G_1)) \cdot h_n^{\underline\pi^2}(\cox(\mathcal G_2)) \\
  &\le C^{n\log\log n} (n!)^{\gamma(\mathcal G_1) - \frac 1{2n} \sum_{u \in S_1} \pi_{u, 2}} \cdot (n!)^{\frac d 2 - \frac 1{2n} \sum_{u \in S_2} \pi_{u, 2}} \\
  &= C^{n\log\log n} (n!)^{\gamma(\mathcal G_1) + \frac d 2 - \frac 1{2n} \sum_{u \in S} \pi_{u, 2}}
  \end{align*}
  where the inequality on the second line follows from the induction hypothesis applied to $\mathcal G_1$ and $\mathcal G_2$. It remains to see that
  \[
  \gamma(\mathcal G_0) = \gamma(\mathcal G_1) + \frac d 2. 
  \]
  To do this we first note that any clique collection containing $w$ cannot maximise $w$: if $\mathcal C = \mathcal C_1 \cup \cdots \cup \mathcal C_r$ is one such collection with $w\in \mathcal C_1$ then we have
  \[
  w(\{v_1\} \cup \cdots \cup \{v_d\} \cup \mathcal C_2 \cup \cdots \cup \mathcal C_r) = w(\mathcal C) - \left(1- 2^{-|\mathcal C_i|}\right) + \frac d 2 > w(\mathcal C).
  \]
  So if $\mathcal C$ maximises $w$ it is of the form $\mathcal C_1 \cup \mathcal C_2$ where $\mathcal C_1$ is a clique in $\mathcal G_1$ (necessary maximising $w$ there) and $\mathcal C_2$ in $\mathcal G_2$ is a maximising clique in $\mathcal G_2 \setminus \{w\}$. The latter has weight $d/2$ and the equality we were after follows. 

  Finally we deal with the case \ref{one_leaf}. In this case $w$ is a leaf of $\mathcal G_0' = \mathcal G_0 \setminus \{v\}$, and we let
  \[
  S' = (S \setminus \{v\})\cup \{w\}.
  \]
  For $\sigma_v \in \sym_n$ is an involution we also define $\underline\pi(\sigma) \in (\mathcal P_{\le 2}(n)^{S'}$ by
  \begin{equation} \label{new_constraints}
  \pi(\sigma)_u = \begin{cases}
             \pi_u \text{ if } u \in S \\
             \pi_\sigma \text{ if } u = w.
             \end{cases}
  \end{equation}
  where $\pi_\sigma$ is the partition into cycles of $\sigma$. Then we have, using Lemma \ref{1vert} and the induction hypothesis, and the fact that $\gamma(\mathcal G_0) = \gamma(\mathcal G_0') + 1/4$: 
  \begin{align*}
  h_n^{\underline\pi}(\cox(\mathcal G_0)) &= \sum_{\sigma \in \stab(\pi_v)} h_n^{\underline\pi(\sigma)}(\cox(\mathcal G_0')) \\
  &\le \sum_{k=0}^n N_{k, \pi_v}^1(n) C^{n\log\log(n)} (n!)^{\gamma(\mathcal G_0') - \frac 1{2n} \left( \frac {n-k}{4n} + \sum_{u \in S', u \not= w} \pi_{v, 2}\right)} \\
  &\le C^{n\log\log n} \sum_{k=0}^n (n!)^{\frac 1 2 - \frac{\pi_{v,2}}{2n} + \gamma(\mathcal G_0') - \frac k{4n} + \frac 1{2n} \left( \frac {n-k}{4n} + \sum_{u \in S', u \not= w} \pi_{v, 2}\right)} \\
  &= C^{n\log\log n} \sum_{k=0}^n (n!)^{\gamma(\mathcal G_0') + \frac 1 4 - \frac 1{2n} \sum_{u \in S} \pi_{u, 2}} \\
  &\le C^{n\log\log n} (n!)^{\gamma(\mathcal G_0) - \frac 1{2n} \sum_{u \in S} \pi_{u, 2}}
  \end{align*}
  which finishes the proof in this case. 
\end{proof}

%%%%%%%%%%%%%%%%%%%%%%%%%%%%%%

\subsection{Gluing vertices}

\begin{prp} \label{gluvert}
  If there exists vertices $w_1, \ldots, w_s$ which are pairwise nonadjacent and such that
  \[
  \mathcal G_0 = \mathcal G \setminus \{w_1, \ldots, w_s\}
  \]
  is a tree so that all vertices adjacent to one of the $w_i$ are leaves in $\mathcal G_0$. Then $\mathcal G, S$ satisfies the conclusion of the Conjecture \ref{conjecture_cox}.
\end{prp}

\begin{proof}
  We prove this by induction on $s$. When $s=1$ and $\mathcal G$ is not a tree we have that the number of neighbours of $w_1$ is at least two: let $p \ge 2$ denote this number. Let $S' = S \cup \{v: v \sim w_1\}$, then we have 
  \[
  h_n^{\underline\pi}(\mathcal G) = \sum_{\sigma\in\sym_n[2]} h_n^{\underline\pi(\sigma)}(\mathcal G_0)
  \]
  where $\underline\pi(\sigma)$ is defined as in \eqref{new_constraints}. Using Proposition \ref{tree_constraints} it follows that:
  \begin{align*}
  h_n^{\underline\pi}(\mathcal G) &\le C^{n\log\log n} \sum_{k=0}^m (n!)^{\gamma(\mathcal G) - \frac 1{2n} \sum_{v \in S} \pi_{v, 2} - p\frac {n-k}{4n}} (n!)^{\frac 1 2 - \frac k{2n}} \\
  & \le C^{n\log\log n} (n!)^{\gamma(\mathcal G) - \frac 1{2n} \sum_{v \in S} \pi_{v, 2}} (n!)^{(p-2)\left(\frac k{4n} - \frac 1 4\right)} \le C^{n\log\log n} (n!)^{\gamma(\mathcal G) - \frac 1{2n} \sum_{v \in S} \pi_{v, 2}}
  \end{align*}
  which finishes the proof in this case. 

  If $s > 2$ we can use exactly the same argument by replacing Proposition \ref{tree_constraints} by the induction hypothesis. 
\end{proof}

%%%%%%%%%%%%%%%%%%%%%%%%%%%%%%

\subsection{Gluing stars} \label{glustar}

In this subsection we work in the following setting: we assume we have a graph $\mathcal G$ and a vertex $v_0 \in \mathcal G$ whose 1-neighbourhood $N_1(v_0)$ is a star, each leaf of which is adjacent to exactly one vertex of $\mathcal G_0 = \mathcal G \setminus N_1(v_0)$, which is a leaf in $\mathcal G_0$. Moreover we assume that the conclusion of Proposition \ref{coxeter_constraints} is known to hold for $\mathcal G_0$. We will forget about the constraints on leaves of $\mathcal G$ for notational ease in the proof below, but it is clear from the argument that they can be incorporated. 

\subsubsection{Reducing to the small orbit counting}

For any $0 \le m, K \le n$ we can define $S(m, K)$ and $L(n-m, K)$ as in Section \ref{sec_raagsmalllargeorbits}. The the proof of the estimates \eqref{dec_small_large} also works when $\Gamma$ is a Coxeter group, using the second statement in Lemma \ref{lem_centbd2}. The analogue of the estimate \eqref{large}, namely that for $K = \floor{\log n}$ we have 
\[
|L(n-m, K)| \le C^{n\log\log n} (n-m)^{\gamma(\mathcal G)}
\]
can also be proven to hold with no modifications to the argument there (except for the induction hypothesis used). 

So we get that :
\begin{equation} \label{same_step}
  h_n(\cox(\mathcal G)) \le C^{n\log\log(n)} \sum_{m=0}^n ((n-m)!)^{\gamma(\mathcal G)} S(m, K). 
\end{equation}
It remains to estimate the terms $S(m, K)$. First we note that if for $1 \le l \le K$ we define $P(m, l)$ to be the subset of $S(m, K)$ consisting of representations in which the orbits of the generators corresponding to $N_1(v_0)\setminus \{v_0\}$ are all of size exactly $l$ then applying the reasoning leading to \eqref{dec_small_large} at most $\log_2 K$ times we get that
\[
\card{S(m, K)} \le C^{n (\log\log(n))^2} \sum_{\substack{m_1, \ldots, m_K \\ \sum_l m_l = m}} \prod_{l=1}^K \card{P(m_l, l)}.
\]
Here the factor up front is an upper bound for $u(n)^{\log_2(K)}$, where $u:\NN\to \RR$ is the function defined in Lemma \ref{lem_centbd2}. The rest of this section is devoted to the proof of the fact that
\begin{equation} \label{fixed-size}
  \card{P(m, l)} \le C^{m\log\log m} (m!)^{\gamma(\mathcal G)},
\end{equation}
which together with the bounds above, implies Proposition \ref{coxeter_constraints}.

%%%%%%%%%%%%%%%%%%%%%%%%%%%%%%

\subsubsection{Small orbits}

When $l = 1$ we have $P(m, 1) = h_m(\cox(\mathcal G \setminus N_1(v_0))$. We can apply the induction hypothesis to $\mathcal G \setminus N_1(v_0)$ and since $\mu_\cox$ is decreasing in subgraphs we get
\begin{equation} \label{vertex_delete}
  P(m_1, 1) \le C^{m\log\log(m)} (m!)^{\mu_\cox(\mathcal G)}
\end{equation}

Now suppose that $l \ge 2$. Recall that we defined
\[
\mathcal G_0 = \mathcal G \setminus N_1(v_0).
\]
As $\mathcal G$ has no triangles $\mu_\cox(\mathcal H) \in \frac 1 4 \NN$ for any subgraph $\mathcal H \subset \mathcal G$ and hence there are then three cases to consider: 
\begin{enumerate}
\item \label{cas_facile} $\mu_\cox(\mathcal G_0) \leq \mu_\cox(\mathcal G) - 1$ ;

\item \label{cas_difficile1} $\mu_\cox(\mathcal G_0) = \mu_\cox(\mathcal G) - 3/4$ ;

\item \label{cas_difficile2} $\mu_\cox(\mathcal G_0) = \mu_\cox(\mathcal G) - 1/2$.
\end{enumerate}
The case \ref{cas_facile} is dealt with as in \eqref{small}, without using any further hypothesis on the graph since this is exactly the case where we can use the conclusion of Lemma \ref{gap1_indep_nb}. The two remaining cases need more care. 

%%%%%%%%%%%%%%%%%%%%%%%%%%%%%%

\subsubsection{Case \ref{cas_difficile2}} 

Let $v_1, \ldots, v_d$ be the neighbours of $v_0$ in $\mathcal G$. We want to count the representations $\rho : \Gamma \to \sym_m$ such that $\langle \rho(\sigma_{v_i}), i = 1, \ldots, d \rangle$ has exactly $m/l$ orbits each of size $l$. As in the RAAG case (Section \ref{sec_raagsmall}) there are at most 
\begin{equation} \label{count_smallreps}
C^{m\log\log(m)}(m!)^{1 - \frac 1 l}
\end{equation}
such representations. For such a representation $\rho$, $\rho(\sigma_{v_0})$ must permute the orbits and we get that there are at most
\begin{equation} \label{count_v0}
|(\sym_l \wr \sym_{m/l})[2]| \le C^{m\log\log(m)}(m!)^{\frac 1{2l}}
\end{equation}
choices for it.

Now let $w_1, \ldots, w_r$ be vertices of $\mathcal G_0$ such that each $w_i$ is adjacent to one of the $v_1, \ldots, v_d$. For each $i$ we have that $\rho(\sigma_{w_i})$ must preserve the partition $\pi_i \in \mathcal P_{\le l}(m)$ given by the orbits of
\[
\langle \rho(\sigma_{v_j}) : v_j \text{ adjacent to } w_i \rangle
\]
and as we assumed that $\mathcal G_0$ satisfies the conclusion of Proposition \ref{coxeter_constraints} it follows that there are at most
\[
C^{m\log\log(m)} \frac{h_m(\mathcal G_0)}{(m!)^{\frac 1{2m} \sum_{i=1}^r \sum_{j=2}^l (j-1)\pi_{i, j}}}
\]
choices for the restriction of $\rho$ to $\cox(\mathcal G_0)$. On the other hand, since each $v_j$ is adjacent to at least one $w_i$ (otherwise we have at least $\gamma(\mathcal G) \ge \gamma(\mathcal G_0) + 3/4$) and together the $\rho(\sigma_{v_j})$ act transitively on each orbit of size $l$ we see that we must have\footnote{This follows from the general fact that if $X = \bigcup_k X_k$ and $X_k \cap \bigcup_{a\not= k} X_a \not= \emptyset$ for all $k$ then $|X| - 1 \le \sum_k (|X_k| - 1)$. }
\[
\sum_{i=1}^r \pi_{i,2} \ge \frac m l (l-1). 
\]
So we get that the number of choices for the restriction is at most 
\begin{equation} \label{count_rest}
C^{m\log\log(m)} \frac{h_m(\mathcal G_0)}{(m!)^{\frac 1 2 - \frac 1{2l}}}. 
\end{equation}

We can finally estimate $P(m, l)$ by the product of \eqref{count_smallreps}, \eqref{count_v0} and \eqref{count_rest} to obtain :
\[
P(m, l) \le E^{m\log\log(m)} (m!)^{\left(1 - \frac 1 l \right) + \frac 1{2l} - \left(\frac 1 2 - \frac 1{2l}\right)} h_m(\mathcal G_0) = E^{m\log\log(m)} (m!)^{\frac 1 2} h_m(\mathcal G_0)
\]
and by the induction hypothesis and the fact that $\mu_\cox(\mathcal G_0) = \mu_\cox(\mathcal G) - 1/2$ we get that
\begin{equation} \label{est_P1}
  P(m, l) \le E^{m\log\log(m)} (m!)^{\mu_\cox(\mathcal G)}
\end{equation}
which finishes the proof in this case.

%%%%%%%%%%%%%%%%%%%%%%%%%%%%%%

\subsubsection{Case \ref{cas_difficile1}} 

This case differs from the preceding in that we may have at most one vertex among the neighbours $v_1, \ldots, v_d$ of $v_0$ which is not adjacent to any leaf of $\mathcal G_0$. Thus in the sequel we will assume that $v_2, \ldots, v_d$ each have a neighbour $v_i \in \mathcal G_0$ such that $\gamma(\mathcal G_0 \setminus w_i) < \gamma(\mathcal G_0)$, but not $v_1$.

Thus we can now only apply the conclusion of Proposition \ref{coxeter_constraints} to the partitions $\pi_2, \ldots, \pi_d$. We can still prove the following inequality (the proof wil be given in the next paragraph)
\begin{equation} \label{dizufzg}
  \sum_{i=1}^r \pi_{i,2} \ge \frac m l \cdot \floor{\frac{l-1} 2}. 
\end{equation}
Since \eqref{count_smallreps} and \eqref{count_v0} remain valid, and $\floor{\frac{l-1} 2} \ge \frac{l-2}2$  we get that
\[
P(m, l) \le E^{m\log\log(m)} (m!)^{\left(1 - \frac 1 l \right) + \frac 1{2l} - \frac 1 2\left(\frac 1 2 - \frac 1 l\right)} h_m(\mathcal G_0) =  E^{m\log\log(m)} (m!)^{\frac 3 4} h_m(\mathcal G_0)
\]
which allows us to conclude.

It remains to prove \eqref{dizufzg}. Let $X$ be a set of cardinality $l$ (one of the orbits of $\rho(\langle \sigma_{v_1}, \ldots, \sigma_{v_d}\rangle)$). It equals the union of the orbits $X_1, \ldots, X_p$ of $\rho(\langle \sigma_{v_2}, \ldots, \sigma_{v_d}\rangle)$ and of the 2-cycles of $\rho(\sigma_{v_1})$. By the case treated above we have that
\[
\sum_{i=1}^r \pi_{i, 2} \ge \sum_k (|X_k| - 1). 
\]
Now $\rho(\sigma_{v_1})$ has at most $l/2$ 2-cycles and it follows that
\[
\frac l 2 + \sum_k (|X_k| - 1) \ge l - 1
\]
which finishes the proof when we apply it to all orbits and sum the results.

%%%%%%%%%%%%%%%%%%%%%%%%%%%%%%%%%%%%%%%%%%%%%%%%%%%%%%%%%%%%%%%%%%%%%%%%%%%%%%%%

\appendix

\section{Lemmas on symmetric groups} \label{appendix}

\subsection{Block decompositions of permutation groups} \label{blockdec}

For $n \in \NN$ we will denote by $[n]$ the set of integers between $1$ and $n$. 

Recall that a permutation group $G \subset \sym_n$ is said to be {\em primitive} if there does not exist disjoint nonempty subsets $\Omega_1, \ldots, \Omega_r$ in $[n]$ of which at least one has cardinality greater than $1$ such that $g\Omega_i \in \{\Omega_1, \ldots, \Omega_r\}$ for all $g \in G$. If there exist such sets $\Omega_i$ and in addition they cover $[n]$ then they are said to form a block system for $G$. 

If $G$ is transitive (but not necessarily primitive) all blocks $\Omega_i$ need to be of the same size and hence  $r|n$. We get a morphism $G \to \sym_r$ from the action of $G$ on $\{\Omega_1, \ldots, \Omega_r\}$. Let $G'$ be its image and $H$ the stabiliser of $\Omega_1$ in $G$. We identify each $\Omega_i$ with $[d]$ (where $d=n/r$) so that $H \subset \sym_d$. Then $G$ is isomorphic to the wreath product $G' \wr H$ that is $G \cong G' \rtimes H^{n/d}$ and we can identify the action of $G$ on $[n]$ with that of $G' \wr H$ where $(g', (h_1, \ldots, h_{n/d}))$ acts by
\[
(g', (h_1, \ldots, h_{n/d})) \cdot x = \left(l_{g'(i)}^{-1} \circ h_i \circ l_i\right) (x)
\]
for $x \in \Omega_i$, where $l_j$ is the bijection from $\Omega_j$ to $[d]$.

The group $H \subset \sym_r$ is primitive if and only if the block system is minimal among all block systems for $G$. There always exists such a system and it follows that any transitive subgroup $G \subset \sym_n$ can be written as $G' \wr H$ where $r|n$ and $r > 1$, $d = n/r$, $G' \subset \sym_r$ is a transitive subgroup and $H \subset \sym_d$ is a primitive subgroup. 

%%%%%%%%%%%%%%%%%%%%%%%%%%%%%%%%%%%%%%%%%%%%%%%%%%%%%%%%%%%%

\subsection{The number of involutions in a permutation group}

For a group $G$ we denote by $G[2]$ be the subset of involutions in $G$ (including the identity, so in particular it's never empty). The following proposition shows that the number if involutions in a permutation group is, at the scale which interests us, equivalent to the square root of the order of the group (better bounds which are close to being sharp and are proven using a completely different approach are available at least in one direction, see \cite{Robinson_MO}). 

\begin{prp} \label{nb_involutions}
  There exists a function $u : \NN \to \RR$ such that 
  \[\log u(n) = O(n\log\log(n))\] 
  as $n\to\infty$ so that for every $n \ge 1$ and every permutation group $G \subset \sym_n$ we have
  \[
  u(n)^{-1} |G|^{\frac 1 2} \le |G[2]| \le u(n) |G|^{\frac 1 2}.
  \]
\end{prp}

\begin{proof}
  It is easy to see that we may assume that $G$ is transitive. Suppose not, then we can write $G = G_1 \times \cdots G_k$ where each $G_i \subset \sym_{n_i}$ is transitive. Applying the argument below to each $G_i$ we get that $G_i[2] \le u(n_i) |G_i|^{\frac 1 2}$. We will see that $u$ is submultiplicative and hence it follows that
  \[
  |G[2]| = \prod_i |G_i[2]| \le \prod_i u(n_i) |G_i|^{\frac 1 2} \le u(n) |G|^{\frac 1 2}
  \]
  and a similar argument applies to deduce the lower bound for $|G[2]|$ from that for the $|G_i[2]|$.

  We first deal with the case where $G \subset \sym_m$ is primitive. In \cite{PraSax}, Praeger and Saxl prove that this implies that either $\card{G} \leq 4^m$ or $G$ is either $\alt_m$ or $\sym_m$ \footnote{In fact, sharper bounds are available, but the above suffices for us. See for instance \cite[Theorem 16.4.1]{Lubotzky_Segal} and references therein.} Thus in the first case there is a $C >1$, independent of $m$ and $G$, such that
  \begin{equation} \label{primitive}
  C^{-m}|G'|^{1/2} \le 1 \le |G'[2]| \le |G'| \le C^m |G'|^{1/2}.
  \end{equation}
  Moreover, in \cite{CHM}, Chowla, Herstein and Moore prove that the number of involutions in $\sym_m$ is $\sim (m!)^{1/2}e(m)$ where $e$ is subexponential, from which the case of $\alt_m$ readily follows. 

To prove the general case, we proceed by induction over $n$.
  Using the block decomposition of \ref{blockdec} we may assume that there are $d|n$, so that $d\geq 2$, a primitive subgroup $H \le \sym_d$ and a transitive subgroup $G' \le \sym_{n/d}$ such that $G = G' \wr H$.   
  
  The morphism $G \to G'$ restricts to a map $\pi : G[2] \to G'[2]$ and we have
  \[
  |G[2]| \le |G'[2]| \cdot \max_{g' \in G[2]} |\pi^{-1}\{ g' \}|.
  \]
  Let $g' \in G'[2]$ have $f$ fixed points. We will assume that $g'$ fixes the blocks $\{1, \ldots, f\}$ and transposes every subsequent pair of blocks. Note that
  \begin{itemize}
  \item within the blocks that $g'$ fixes, every element of  $\pi^{-1}\{ g' \}$ needs to act as an involution.
  \item within two blocks that are permuted by $g'$, every element of  $\pi^{-1}\{ g' \}$ can act by an arbitrary element of $H$. However, because of the wreath product structure, these two elements will be each others inverses. Figure \ref{pic_blocks} offers a graphical depiction of this situation.
  \end{itemize}
  
  \begin{figure}[H]
    \begin{center}
      \begin{overpic}[scale=1]{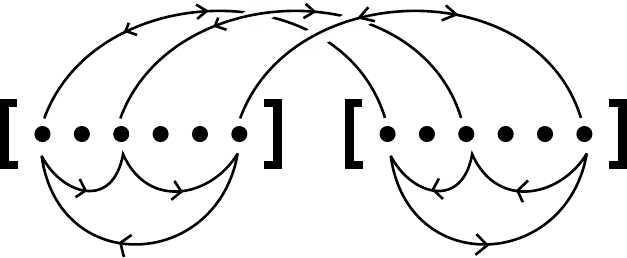} 
        \put (48,42) {$g'$}
        \put (20,-5) {$h$}
        \put (75,-5) {$h^{-1}$}
      \end{overpic}
    \end{center}
    \caption{Part of the action of $G'\wr H$ on two blocks. Because $h$ acts by a $3$-cycle on three elements in the block on the left hand side, it acts in the opposite direction on their images by $g'$ in the block on the right hand side.}
    \label{pic_blocks}
  \end{figure} 
  
  In symbols this means that an element in $\pi^{-1}\{ g' \}$ is necessarily of the form
  \[
  (g', (h_1, \ldots, h_f, h_{f+1}, h_{f+1}^{-1}, \ldots, h_{(n/d - f)/2}, h_{(n/d - f)/2}^{-1}))
  \]
  where $h_1, \ldots, h_f$ are involutions in $H$ and $h_{f+1}, \ldots, h_{(n/d - f)/2}$ are arbitrary elements of $H$. Thus there are exactly
  \[
  |H[2]|^f \cdot |H|^{\frac{n/d - f} 2}
  \]
  elements in $\pi^{-1}\{g'\}$. Since $H$ is primitive we have that
  \[
  C^{-d} |H|^{1/2} \le |H[2]| \le C^d |H|^{1/2}
  \]
  so that
  \[
  C^{-fd} |H|^{f/2} \cdot |H|^{\frac{n/d - f} 2} \le |\pi^{-1}\{ g' \}| \le C^{fd} |H|^{f/2} \cdot |H|^{\frac{n/d - f} 2} 
  \]
  which in turn yields
  \[
  C^{-n} |H|^{\frac{n/d} 2} \cdot |G'[2]| \le |G[2]| \le C^n  |H|^{\frac{n/d} 2} \cdot |G'[2]| 
  \]
  and using the induction hypothesis (and the fact that $d\geq 2$) it follows that 
  \begin{equation} \label{HR}
    C^{-n} |H|^{\frac{n/d} 2} \cdot u(n/d)^{-1} |G'|^{\frac 1 2} \le |G[2]| \le C^n  |H|^{\frac{n/d} 2} \cdot u(n/d) |G'|^{\frac 1 2}. 
  \end{equation}
  Now since we have
  \[
  |G| = |G'| \cdot |H|^{n/d}
  \]
  we get the desired bounds from \eqref{HR} and \eqref{primitive}, if we can prove that we can find $u$ such that 
  \begin{equation} \label{bad}
    C^n u(n/d) \le u(n) 
  \end{equation}
  for large enough $n$. An elementary computation shows that the function 
  \[u(n) = C^{n\log\log(n+\exp(\exp(2)))},\]
  for example, satisfies this property. 
\end{proof}

%%%%%%%%%%%%%%%%%%%%%%%%%%%%%%%%%%%%%%%%%%%%%%%%%%%%%%%%%%%%

\subsection{Estimates for centralisers}

\begin{lem} 
\label{lem_transitive}
Let $H$ be a transitive subgroup of $\sym_n$. Then $\card{Z_{\sym_n}(H)} \leq  n$.  
\end{lem} 

\begin{proof}
  This follows immediately from \cite[Theorem 4.2A(i)]{DM}. 
\end{proof} 

\begin{lem}\label{lem_centbd2}
  Let $0\leq k\leq n\in\mathbb{N}$ and $\sym_k\times\sym_{n-k}\simeq H<\sym_n$. Then for any $G<\sym_n$ we have
  \[
  \card{\cent_{\sym_n}(G)} \leq 2^n \cdot \card{\cent_H(G)}.
  \]
  In addition, if $u$ is the function from Theorem \ref{nb_involutions} then
  \[
  \card{\cent_{\sym_n}(G)[2]} \leq 2^n u(n)^2 \cdot \card{\cent_H(G)[2]}.
  \]  
\end{lem}

\begin{proof}
  We have
  \[
  \card{\cent_{\sym_n}(G)} = \card{\cent_{\sym_n}(G)/\cent_H(G)} \cdot \card{\cent_H(G)}.
  \]
  There is well-defined injective map
  \[
  \cent_{\sym_n}(G)/\cent_H(G) \to \sym_n/H.
  \]
  As such
  \[
  \card{\cent_{\sym_n}(G)/\cent_H(G)} \leq \card{\sym_n/H} = \binom{n}{k} \leq 2^n,
  \]
  which proves the first statement. The second one then immediately follows from this and Theorem \ref{nb_involutions}. 
\end{proof}

%%%%%%%%%%%%%%%%%%%%%%%%%%%%%%%%%%%%%%%%%%%%%%%%%
%		B I B L I O G R A P H Y
%%%%%%%%%%%%%%%%%%%%%%%%%%%%%%%%%%%%%%%%%%%%%%%%%
\bibliography{bib}{}
\bibliographystyle{amsplain}

\end{document}